\documentclass[12pt]{amsart}
\usepackage{amscd,amsthm,amssymb,amsfonts,amsmath,euscript}
\usepackage{mathrsfs}
\usepackage{comment}
\theoremstyle{plain}
\newtheorem{thm}{Theorem}[section]
\newtheorem{lemma}[thm]{Lemma}
\newtheorem{prop}[thm]{Proposition}
\newtheorem{cor}[thm]{Corollary}

\theoremstyle{definition}
\newtheorem{defn}[thm]{Definition}

\theoremstyle{remark}
\newtheorem{remark}[thm]{Remark}
\newtheorem*{thank}{{\bf Acknowledgments}}
\newcommand{\nc}{\newcommand}

\def\makeop#1{\expandafter\def\csname#1\endcsname
  {\mathop{\rm #1}\nolimits}\ignorespaces}
\makeop{Hom}   \makeop{End}   \makeop{Aut}   \makeop{Isom}  \makeop{Pic} 
\makeop{Gal}   \makeop{ord}   \makeop{Char}  \makeop{Div}   \makeop{Lie} 
\makeop{PGL}   \makeop{Corr}  \makeop{PSL}   \makeop{sgn}   \makeop{Spf}
\makeop{Spec}  \makeop{Tr}    \makeop{Nr}    \makeop{Fr}    \makeop{disc}
\makeop{Proj}  \makeop{supp}  \makeop{ker}   \makeop{im}    \makeop{dom}
\makeop{coker} \makeop{Stab}  \makeop{SO}    \makeop{SL}    \makeop{SL}
\makeop{Cl}    \makeop{cond}  \makeop{Br}    \makeop{inv}   \makeop{rank}
\makeop{id}    \makeop{Fil}   \makeop{Frac}  \makeop{GL}    \makeop{SU}
\makeop{Nrd}   \makeop{Sp}    \makeop{Tr}    \makeop{Trd}   \makeop{diag}
\makeop{Res}   \makeop{ind}   \makeop{depth} \makeop{Tr}    \makeop{st}
\makeop{Ad}    \makeop{Int}   \makeop{tr}    \makeop{Sym}   \makeop{can}
\makeop{length}\makeop{SO}    \makeop{torsion} \makeop{GSp} \makeop{Ker}
\makeop{Adm}   \makeop{Mat}
\makeop{Q-isog}
\makeop{Rad}
\makeop{Ind}

\DeclareMathOperator{\Irr}{Irr}
\DeclareMathOperator{\GU}{GU}

\def\makebb#1{\expandafter\def
  \csname bb#1\endcsname{{\mathbb{#1}}}\ignorespaces}
\def\makebf#1{\expandafter\def\csname bf#1\endcsname{{\bf
      #1}}\ignorespaces} 
\def\makegr#1{\expandafter\def
  \csname gr#1\endcsname{{\mathfrak{#1}}}\ignorespaces}
\def\makescr#1{\expandafter\def
  \csname scr#1\endcsname{{\EuScript{#1}}}\ignorespaces}
\def\makecal#1{\expandafter\def\csname cal#1\endcsname{{\mathcal
      #1}}\ignorespaces} 
\def\doLetters#1{#1A #1B #1C #1D #1E #1F #1G #1H #1I #1J #1K #1L #1M
                 #1N #1O #1P #1Q #1R #1S #1T #1U #1V #1W #1X #1Y #1Z}
\def\doletters#1{#1a #1b #1c #1d #1e #1f #1g #1h #1i #1j #1k #1l #1m
                 #1n #1o #1p #1q #1r #1s #1t #1u #1v #1w #1x #1y #1z}
\doLetters\makebb   \doLetters\makecal  \doLetters\makebf
\doLetters\makescr 
\doletters\makebf   \doLetters\makegr   \doletters\makegr

\normalsize

\makeop{Bl}

\def\Sh{{\rm Sh}}

\newcommand{\Z}{\mathbb Z}
\newcommand{\Q}{\mathbb Q}
\newcommand{\R}{\mathbb R}
\newcommand{\C}{\mathbb C}
\newcommand{\D}{\mathbb D}    % pro algebraic torus
\newcommand{\A}{\mathbb A}    % for adele
\newcommand{\F}{\mathbb F}

\nc{\embed}{\hookrightarrow}
\newcommand{\dieu}{Dieudonn\'{e} }

\nc{\ol}{\overline}
\nc{\wt}{\widetilde}
\nc{\opp}{\mathrm{opp}}

\makeop{Ram}
\makeop{Rep}

\begin{document}
\renewcommand{\thefootnote}{\fnsymbol{footnote}}
\setcounter{footnote}{-1}
\numberwithin{equation}{section}
%\gammamberwithin{section}{chapter}

%\usepackage[notref,notcite]{showkeys}

\title[Hecke eigensystems of automorphic forms (mod $p$)]{Hecke eigensystems of automorphic forms (mod $p$) of Hodge type and algebraic modular forms}
 \author{Yasuhiro Terakado}
\address{
(Terakado)Institute of Mathematics, Academia Sinica \\
Astronomy Mathematics Building \\
No.~1, Roosevelt Rd. Sec.~4 \\ 
Taipei, Taiwan, 10617} 
\email{terakado@gate.sinica.edu.tw} 
  
\author{Chia-Fu Yu}
\address{
(Yu)Institute of Mathematics, Academia Sinica \\
Astronomy Mathematics Building \\
No.~1, Roosevelt Rd. Sec.~4 \\ 
Taipei, Taiwan, 10617} 
\email{chiafu@math.sinica.edu.tw}

\address{
(Yu)National Center for Theoretical Sciences
No.~1  Roosevelt Rd. Sec.~4,
National Taiwan University
Taipei, Taiwan, 10617}

%\date{June 23, 2000}

\date{\today}
\subjclass[2010]{} 
\keywords{}

%\subjclass[2000]{11R52,11R58}
\keywords{Shimura varieties, 
automorphic forms$\pmod p$, 
Hecke eigensystems, algebraic modular forms}  

\begin{abstract}
We show that the 
systems of prime-to-$p$ Hecke eigenvalues arising from automorphic forms$\pmod p$   for a good prime $p$ associated to an algebraic group $G/\mathbb Q$ of Hodge type are 
the same as those arising from 
algebraic modular forms$\pmod p$  associated to an inner form of $G$. 
\end{abstract} 

\maketitle 
\tableofcontents
\section*{Introduction}
\label{sec:I}
In his letter \cite{Serre} to Tate, Serre proved  that the 
 systems of prime-to-$p$ Hecke eigenvalues arising from modular forms$\pmod p$  are the same as those arising from locally constant functions 
\[f : D^{\times} \backslash (D\otimes_{\mathbb Q}\mathbb A_f)^{\times} \to \overline{\F}_p,\]
 where $D$ is the quaternion $\mathbb Q$-algebra ramified precisely at a prime $p$ and $\infty$, and $\mathbb A_f$ is the ring of finite ad\`eles of the rational numbers $\Q$. 
 This correspondence was obtained from restricting modular forms to the supersingular locus of the modular curve modulo $p$, and 
  the algebra $D$ appeared 
   as the endomorphism algebra of a supersingular elliptic curve. 
    This result can also be regarded  as a mod $p$ analogue of the Jacquet-Langlands  correspondence. 
 
  Gross  was inspired by this work and introduced  the notion of algebraic modular forms$\pmod p$ on reductive algebraic groups over $\mathbb Q$ satisfying certain conditions \cite{Gross}. 
  By using this notion, 
Ghitza  generalized  Serre's result to Siegel modular forms$\pmod p$ of genus $g>1$ in \cite{Ghitza}.
 He considered algebraic modular forms on an inner form of $\GSp_{2g}$ over $\Q$ which  appears as the automorphism group of a superspecial abelian variety. 
Here, an abelian variety is  called superspecial  if it is  isomorphic to a product of supersingular elliptic curves. 
 Further, Reduzzi investigated the above correspondence for Shimura varieties of PEL-type  in \cite{Reduzzi}, and    
  gave full details for the unitary  Shimura variety associated an imaginary quadratic field.
 He considered algebraic modular forms on the inner form  associated to a  superspecial point,  assuming that the superspecial locus of the Shimura variety (i.e.  the set of $\overline{\F}_p$-points whose underlying abelian varieties are superspecial) 
 is non-empty.  
 In \cite{GK}, Goldring and Koskivirta generalized the first step of Serre's work to arbitrary Shimura varieties of  Hodge type. 
 More precisely, they showed that the systems of Hecke eigenvalues which appear on the whole Shimura variety are the same as those which appear on the unique zero-dimensional Ekedahl-Oort (EO) stratum. 
 We remark that while the superspecial locus  might be empty, 
 the zero-dimensional EO stratum will always be  non-empty.

 In  this paper, we generalize Serre's result to arbitrary Shimura varieties of Hodge type, using the theory of  algebraic modular forms.  
  To state the main result we fix some notations. 
  In this introduction,  assume $p > 2$. 
 Let $(G,X)$ be a Shimura datum of Hodge type. 
 Assume $G$ is unramified at $p$ and 
  ${\sf K}={\sf K}_p{\sf K}^p\subset G(\A_f)$ is an open compact subgroup with a  hyperspecial level ${\sf K}_p$ at $p$.
  We write  $\mathscr S_{\sf K}=\mathscr S_{\sf K} (G,X)$ for  the integral model of the  Shimura variety given by Kisin in  \cite{Kisin1}. 
Let  $L$ be the centralizer   in $G \otimes_{\Z_{(p)}} \overline{\F}_p$ of  a cocharacter induced by the Shimura datum $(G,X)$.  
We choose a maximal torus $T$ of $L$ and     a positive system of roots of $L$, and 
  we  write $X_{+,L}^*(T)$  for 
   the set of dominant weights.          
  For any $\xi \in X_{+,L}^*(T)$, let   $\mathscr V(\xi)$ be the \emph{automorphic bundle of weight $\xi$} on the special fiber  $\mathscr S_{\sf K} \otimes _{\Z_{(p)}} \overline{\F}_p$ ($\S$\ref{autbdl}). 
  For a smooth toroidal compactification $\mathscr S_{{\sf K}}^{\Sigma}$ of $\mathscr S_{{\sf K}}$, let $\mathscr V^{{\rm can}}(\xi)$  
  (resp. $\mathscr V^{{\rm sub}}(\xi)$) be the canonical (resp. subcanonical) extension  of $\mathscr V(\xi)$ ($\S$\ref{ext}). 
  
  Now  we choose a base point $x \in \mathscr S_{\sf K}(\overline{\F}_p)$. 
   We write 
 $\mathcal A_x$ for the fiber of the universal abelian scheme $\mathcal A \to \mathscr S_{\sf K}$ at $x$.  
  Let $\D(\mathcal  A_x)$ be the contravariant Dieudonn\'{e} module associated to the $p$-divisible group of $\mathcal A_x$.  
Let $I$ be the algebraic group over $\Q$ whose points consist of  automorphisms of $\mathcal A_x$ fixing the tensors ($\S$\ref{sunif}). 
 We put  $I(\Z_p):=I(\Q_p)\cap \Aut_{W(\overline{\F}_p)}(\D(\mathcal A_x))$.  
  Then the reduction$\pmod p$ of an element in $I(\Z_p)$ induces an automorphism of $\D(\mathcal A_x) \otimes \overline{\F}_p$ preserving the Hodge filtration, and hence it also induces an automorphism of its graded module. 
Let $U_p$ be the subgroup of $I(\Z_p)$ consisting of elements whose induced automorphisms  of the graded module  are trivial ($\S$\ref{unifbdl}). 
 Let $I(p):=I(\Z_p)/U_p$. 
Further we  identify the group $I \otimes_{\Q} \Q_{\ell}$ as a subgroup of $G_{\Q_{\ell}}$ for $\ell \neq p$, and put $U^p:=I(\A_f^p)\cap {\sf K}^p$. 
  
  For a finite dimensional simple left  $\overline{\F}_p[I(p)]$-module  $V_{\tau} \in {\rm Irr}(I(p))$, 
we define the \emph{space of  algebraic modular forms  on $I$ of  level $U=U_pU^p$ and weight $V_{\tau}$} as follows:  
\begin{align*}
M^{\rm{alg}}(I, V_{\tau}, U) := 
\{ f : I(\mathbb Q) & \backslash I(\mathbb A_f) /U\to V_{\tau} \ \vert  
\ f(x g)=g^{-1}f(x) 
\\
& {\rm for \ all} 
  \ x \in I(\mathbb Q) \backslash I(\mathbb A_f) /U 
  \ {\rm and} 
   \ g\in I(p)\}. 
  \end{align*} 
We consider the group $I$ associated to a point  on the \emph{basic Newton stratum} 
 ($\S$\ref{Newton}).
\begin{thm}\label{intro}
Assume that $x$ is lying on the basic Newton stratum of $\mathscr S_{\sf K} \otimes \overline{\F}_p$. 
If $(G,X)$ is neither of compact-type, nor of PEL-type, 
then we further assume the  condition  $(\star)$ on the boundary of $\mathscr S_{{\sf K}}^{\Sigma}$, stated in Theorem \ref{GK} below. 
Then 
the systems of prime-to-$p$ Hecke eigenvalues appearing in each of the following spaces 
 \begin{align}\label{mf} 
 \bigoplus_{{\xi}\in X_{+,L}^*(T)} H^0(\mathscr S_{\sf K}\otimes_{\Z_{(p)}} \overline{\F}_p, \mathscr {V}(\xi)), 
 \bigoplus_{{\xi}\in X_{+,L}^*(T)} H^0(\mathscr S^{\Sigma}_{\sf K}\otimes_{\Z_{(p)}} \overline{\F}_p, \mathscr {V}^{\bullet}(\xi)) 
 \end{align}
 for $\bullet \in \{{\rm can, 
 sub}\}$ 
 are the same as those appearing in the space 
 \begin{align}\label{algmf}
  \bigoplus_{V_{\tau} \in {\rm Irr}(I(p))}M^{\rm{alg}}(I, V_{\tau}, U) .\end{align}
 
  As a consequence, the set of  those systems of Hecke eigenvalues appearing  the above space (\ref{algmf})  
 is independent of the choice of $x$ lying on the basic Newton stratum. 
\end{thm}
 This result can be regarded as a  generalization of the results   of Serre, Ghitza, and  Reduzzi   to Hodge type Shimura varieties.  
 Further, it  generalizes  the condition on  the base point $x$:  
It considers the basic Newton stratum, which has positive dimension in general, whereas previous results considered at most zero-dimensional subschemes (the superspecial locus or the zero-dimensional EO stratum). 
   In particular,  it is new even in  the Siegel  case. 

The proof consists of four parts. 
First we use the Rapoport-Zink uniformization for Hodge type Shimura varieties 
given by Kisin \cite{Kisin2}, Howard-Pappas \cite{HP}, and Xiao-Zhu \cite{XZ}. 
 Let $x$ be a point in the basic Newton stratum and let  $\mathcal C(x)$ for the \emph{central leaf} passing through $x$ ($\S$\ref{Newton}). 
 Then the uniformization map induces  a $G(\A_f)$-equivariant  bijection  
\[I(\Q) \backslash I(\A_f)/I(\Z_p)U^p\xrightarrow{\sim} \mathcal C(x).\]
In particular, one sees that $\mathcal C(x)$ is finite (Corollary \ref{c=z}).  

Next we construct  an embedding (Proposition \ref{propemb}) 
  \[
   I(\mathbb Q) \backslash I(\mathbb A_f) / U
 \hookrightarrow
 \mathcal L_{\mathcal C(x)}\]
which lifts the above bijection.   Here, 
   $\mathcal L_{\mathcal C(x)}$ is the scheme associated to the restriction $\mathcal L \vert _{\mathcal C(x)}$ of  a  principal $L$-bundle  $\mathcal L$ over 
   $\mathscr S_{{\sf K}}$. 
 By this embedding, we can  interpret the space  of automorphic forms  $H^0(\mathcal C(x), \mathscr V(\xi)\vert_{\mathcal C(x)})$ on the central leaf as the space of some vector valued functions  on the scheme $\mathcal L_{\mathcal C(x)}$ (Lemma \ref{fun=aut}). 
 
 Furthermore, these vector valued functions on $\mathcal L_{\mathcal C(x)}$ can be regarded as algebraic modular forms (Proposition \ref{alg}). 
 These constructions are   compatible with prime-to-$p$ level structures ${\sf K}^p$, so   
  we see that the systems of prime-to-$p$ Hecke eigenvalues appearing in the space
 \begin{align}\label{clmf}
 \bigoplus_{\xi \in X_{+,L}^{*}(T)}H^0(\mathcal C(x), \mathscr V(\xi)\vert_{\mathcal C(x)})
 \end{align} 
 are the same as those  appearing in the space of  algebraic modular forms (\ref{algmf}) (Theorem \ref{Hevalg}).  

Finally, we will see that the systems of Hecke eigenvalues appearing in the  above space (\ref{clmf}) 
are the same as those appearing in the space (\ref{mf}) of automorphic forms   on $\mathscr S_{{\sf K}}\otimes \overline{\F}_p$  (Theorem \ref{GK}). 
 This will be deduced by some  modifications of the results of   Goldring and Koskivirta in \cite{GK}.
  Here, we need  the assumption on the boundary.  
 This is an outline of the proof of   Theorem \ref{intro}.   

As a corollary, one can bound the number of those systems of eigenvalues  in terms of the cardinality of the central leaf $\mathcal C(x)$ and the dimensions of simple $\overline{\F}_p[I(p)]$-modules $V_{\tau}$ (Corollary \ref{number}).
 We apply this corollary to the Shimura variety  associated to the similitude group of a  Hermitian form with a totally indefinite quaternionic multiplication, and give an explicit upper bound for the number of its systems of Hecke eigenvalues (Theorem \ref{intro2}).

\begin{thank}
Part of this work was done during our  visits in Wuhan
University, the Morningside Center of Mathematics, CAS and POSTECH. 
We thank Jiangwei Xue and Xu Shen for their invitations and the
institutions for hospitality and good research conditions.  
This paper relies on results of Professor Mark Kisin on 
Shimura varieties of Hodge
type and the construction of automorphic forms by Professor Kai-Wen
Lan. 
We also thank Professor Benedict Gross
for helpful comments and sharing his preprint \cite{Gross:2019} to
us. 
Further we thank the referee for a  thorough reading of the manuscript  and valuable suggestions. 
 The first author thanks Chao Zhang for answering questions on
stratifications of Hodge type Shimura varieties. 
The second author is
partially supported by the MoST grants 107-2115-M-001-001-MY2 and  109-2115-M-001-002-MY3.
\end{thank} 

\section{Integral models of Shimura varieties of Hodge type}
 From now on, we write $k$ for an  algebraic closure $\overline{\F}_p$ of the finite field $\F_p$. 
 We assume that   $p \neq 2$. 
  Let $W=W(k)$ be the ring of Witt vectors of $k$, 
  and let $K:=W[\frac{1}{p}]$ be the field of fractions of $W$. 
  
  In this section, 
we recall results of Kisin \cite{Kisin1}
 about integral models of Shimura varieties of Hodge type. 
 \subsection{Shimura varieties of Hodge type}
  Let $G$ be a connected reductive group over $\mathbb Q$ and $X=\{h\}$ a $G(\R)$-conjugacy class of maps of algebraic groups over $\mathbb R$ 
\[h : \Res_{\mathbb C/\R}\mathbb G_{m \C} \to G_{\R},\]
such that $(G, X)$ is a Shimura datum \cite[\S 2.1]{Deligne}.
Denote $\mu_h :
 \mathbb G_{m\C} \to G_{\C}$ by $\mu_h(z)=h_{\C}(z,1)$. 
The reflex field $E \subset \overline{\Q} \subset \C$ is the  field of definition of the conjugacy class of $\mu_h$. 

Let $\A_f$ denote the finite adeles over $\Q$, and $\A_f^p \subset \A_f$ the subgroup of adeles with trivial component at a prime $p$. 
Let ${\sf K}={\sf K}_p{\sf K}^p \subset G(\A_f)$ where 
${\sf K}_p \subset G(\Q_p)$, 
and ${\sf K}^p \subset G(\A_f^p)$ are compact open subgroups. 
In the following, we always  assume that ${\sf K}^p$ is sufficiently small. 
Then, by a theorem of Baily-Borel,  
\[ \Sh_{\sf K}(G,X)_{\C} := G(\Q) \backslash X \times G(\A_f)/{\sf K} \]
has a natural structure of an algebraic variety over $\C$. 
 We write  
$\Sh_{\sf K}(G,X)$ for Deligne's canonical model 
over 
$E$ \cite[\S 2.2]{Deligne}. 

We fix a finite dimensional $\Q$-vector space $V$ equipped with a perfect alternating pairing $\psi$. 
Let $\GSp=\GSp(V,\psi)$ denote the corresponding group of symplectic similitudes, and let $S^{\pm}$ be the Siegel upper and lower half spaces, 
defined as the set of maps $h: \Res_{\C/\R}\mathbb G_{m \C} \to \GSp_{\R}$ 
such that 
\begin{itemize}
\item[(1)] The $\C^{\times}$ action on $V_{\R}$ gives rise to a Hodge structure of type $(-1, 0), (0,-1) : $
\[ V_{\C} \cong V^{-1,0} \oplus V^{0, -1}.\]
\item[(2)] The pairing $V_{\R} \times V_{\R} \to \R$ defined by $(x,y) \mapsto \psi (x, h(i)y)$ is (positive or negative) definite. 
\end{itemize}

We call  $(G,X)$ a Shimura datum of \emph{Hodge type} if there is an embedding of algebraic groups
\begin{align}\label{iota}\varphi : G \hookrightarrow \GSp
\end{align}
over $\Q$ inducing a morphism of Shimura data $(G, X) \to (\GSp, S^{\pm})$.   
Henceforth, we assume that $(G,X)$ is of Hodge type.

For the rest of this paper, we assume that $G$ is unramified at $p$. 
Then $G$ extends to a connected reductive group over $\Z_{(p)}$, which we again  denote by $G$.  
By \cite[2.3.1]{Kisin1}, the  embedding $\varphi$ is induced by an embedding $G \to  \GL(\Lambda_{\Z_{(p)}})$ for some $\Z_{(p)}$-lattice $\Lambda_{\Z_{(p)}} \subset V$. 
By Zarhin's trick, after replacing $\Lambda_{\Z_{(p)}}$ by $\Hom_{\Z_{(p)}}(\Lambda_{\Z_{(p)}}, \Lambda_{\Z_{(p)}})^{\oplus4}$, we may assume that  $\Lambda_{\Z_{(p)}}$ is self-dual with respect to $\psi$. 
We then have a closed immersion  of reductive group schemes
\[\varphi : G  \hookrightarrow \GSp(\Lambda_{\Z_{(p)}}, \psi\vert_{\Lambda_{\Z_{(p)}}})\] 
over $\Z_{(p)}$ with generic fiber (\ref{iota}). 
We will write $\Lambda_{\Z_{(p)}}^{*}$ for  the dual of $\Lambda_{\Z_{(p)}}$. 
Further, for each $\Z_{(p)}$-algebra $R$, we write $G_{R}$ (resp. $\Lambda_{R}$, $\Lambda^*_R$) for $G \otimes_{\Z_{(p)}} R$ (resp. $\Lambda_{\Z_{(p)}} \otimes _{\Z_{(p)}} R, \Lambda^{*}_{\Z_{(p)}} \otimes_{\Z_{(p)}} R$). 

Let  ${\sf K}_p$ and ${\sf K}'_p$ be  the  subgroups  
 \begin{align*}  
  {\sf K}_p  
 & :=G(\Z_p)
  \subset G(\Q_p), 
  \\ 
  {\sf K}'_p 
   & :=
  \GSp 
  (
  \Lambda_{\Z_{(p)}}, 
  \psi  \vert _{\Lambda_{\Z_{(p)}}}
  )(\Z_p)
  \subset \GSp(\Q_p).
  \end{align*} 
By \cite[2.1.2]{Kisin1}, for each ${\sf K}^p$ there exists a compact open subgroup 
${\sf K}'^p \subset \GSp(\A_f^p)$ containing ${\sf K}^p$ and such that $\varphi$ induces an embedding 
\[\varphi : \Sh_{\sf K}(G, X) \hookrightarrow  \Sh_{{\sf K}'}(\GSp, S^{\pm})\]
of $E$-schemes, where 
${\sf K}':={\sf K}'_p {\sf K}'^p \subset \GSp(\A_f)$ and $\Sh_{{\sf K}'}(\GSp,S^{\pm})$ is the Siegel Shimura variety.  
\subsection{Integral models}\label{intmdl}
Let $S$ be a $\Z_{(p)}$-scheme and $ A \to S$ be an abelian scheme. 
The prime-to-$p$ Tate module is defined by 
\[\widehat{T}^p(A)=\lim _{p\nmid n} A[n],\]
viewed as an \'etale local system on $S$. 
Write $\widehat{V}^p( A)=\widehat{T}^p(A) \otimes_{\Z}\Q.$ 
 
 The category of \emph{abelian  schemes over $S$ up to prime-to-$p$ isogeny} is defined as follows: 
 The objects $A$ is an abelian scheme over $S$. 
 A morphism $f: A_1 \to  A_2$ is an element of the module $\Hom_S( A_1,  A_2)\otimes_{\mathbb Z}\mathbb Z_{(p)}$. 
  An isomorphism in this category will be called a prime-to-$p$ quasi-isogeny. 
 If $A$ is an abelian scheme up to prime-to-$p$ isogeny,  we write  ${A^t}$ for the dual abelian scheme up to prime-to-$p$ isogeny. 
 
 A $\Z_{(p)}^{\times}$-\emph{polarization} of $A$ is a  prime-to-$p$ quasi-isogeny  $\lambda :  A \to { A}^t$ such that $n\lambda$  for some positive integer $n$   is induced by an ample line bundle on $A$.  

Let ${\sf K}'^{p} \subset \GSp(\A_f^p)$ be any compact open subgroup. 
Assume that $(A, \lambda)$ is an abelian scheme with a $\Z_{(p)}^{\times}$-polarization. 
Denote by $\underline{\Isom}(\Lambda_{\A_f^p}, \widehat{V}^p(A))/{\sf K}'^p$ 
 the \'etale sheaf on $S$ consisting of ${\sf K}'^p$-orbits of isomorphisms 
 $\Lambda_{\A_f^p} \xrightarrow{\sim} \widehat{V}^p( A)$
  which are compatible with the parings induced by 
  $\psi$ and 
  $\lambda$, up to a scalar in   
  $(\A_f^{p})^{\times}$. 
A \emph{${\sf K}'^{p}$-level structure} on $(A, \lambda)$ is a global section
\[\eta_{{\sf K}'} \in \Gamma(T, \underline{\Isom}(\Lambda_{ \A_f^p}, \widehat{V}^p(A))/{\sf K}'^{p}).\]

In the following, we assume that ${\sf K}'^p$ is sufficiently small. 
By \cite[\S 5]{Kottwitz}, 
 the functor which assigns to $S$ the set of isomorphism classes of triples $( A, \lambda, \eta)$ as above, is representable by a smooth $\Z_{(p)}$-scheme $\mathscr S_{{\sf K}'}(\GSp, S^{\pm})$, such that one has a natural identification 
\[\mathscr S_{{\sf K}'}(\GSp, S^{\pm}) \otimes_{\Z_{(p)}} 
\Q \xrightarrow{\sim}
\Sh_{{\sf K}'}(\GSp, S^{\pm}).\]

 Let $O$ be the ring of integers of $E$. 
 We write  $O_{(\mathfrak p)}$ for the localization of $O$ at the prime $\mathfrak p$ corresponding to an fixed embedding $\overline{\Q} \to \overline{\Q}_p$, and write  $E_{\mathfrak p}$ for the completion of $E$ at $\mathfrak p$. 
 We write $\mathscr{S}_{\sf K}(G, X)^{-}$ for the Zariski closure of $\Sh_{\sf K}(G, X)$ in $\mathscr S_{{\sf K}'}(\GSp, S^{\pm}) \otimes _{\Z_{(p)}}  O_{(\mathfrak p)}$. 
  Then the canonical smooth model 
  $\mathscr S_{{\sf K}}=\mathscr S_{\sf K}(G,X)$ over $O_{(\mathfrak p)}$ is constructed as the normalization of $\mathscr S_{\sf K}(G, X)^{-}$. 
 
 For any compact open ${\sf K}^p_1 \subset {\sf K}^p$, we set   
${\sf K}_1:={\sf K}_p{\sf K}_1^p$. 
Then there is a finite \'{e}tale projection $\pi_{{\sf K}_1/{\sf K}} : \mathscr S_{{\sf K}_1} \to \mathscr S_{\sf K}$.
Further, for $h \in G(\A_f^p)$ 
 there is a right action 
 $h : \mathscr S_{\sf K} \to \mathscr S_{h^{-1}{\sf K}h}$.

Denote by $\Lambda_{\Z_{(p)}}^{\otimes}$ the direct sum of all the $\Z_{(p)}$-modules which can be formed from $\Lambda_{\Z_{(p)}}$ using the operations of taking duals, tensor products, symmetric powers and exterior powers. 
It satisfies that $\Lambda_{\Z_{(p)}}^{\otimes} \xrightarrow{\sim} \Lambda_{\Z_{(p)}}^{*\otimes}$ so that a tensor in the left hand side may be regarded in the right hand side. 
 By \cite[1.3.2]{Kisin1}, the subgroup $G
 \subset \GL(\Lambda_{\Z_{(p)}})$ is the scheme theoretic stabilizer of a collection of tensors $(s_{\alpha})\subset \Lambda_{\Z_{(p)}}^{\otimes}$.  
 
 Let $\mathcal A' \to \mathscr S_{{\sf K}'}(\GSp, S^{\pm})$ be the universal  abelian scheme and let  
$\mathcal A \to \mathscr S_{\sf K}$ be its pullback  via $\varphi$.   
Write $H^1_{{\rm dR}}(\mathcal A/\mathscr S_{\sf K})$ for the first relative  de Rham cohomology. 
By \cite[Cor.~2.3.9]{Kisin1}, 
there are \emph{de Rham tensors} 
\[t_{\alpha, {\rm dR}} \in H^1_{{\rm dR}}(\mathcal A/\mathscr S_{\sf K})^{\otimes},\]
 defined as sections of a coherent $\mathcal O_{\mathscr S_{\sf K}}$-module. 
 
\section{Mod $p$ points of  the  Rapoport-Zink uniformization}
  In this subsection, we review the construction of  $k$-values of the Rapoport-Zink   uniformization given by Kisin in \cite{Kisin2}.   
  This will be used in $\S$\ref{unifbdl} with some modification. 
  
  From now on, we write $S_{\sf K}$ for the special fiber $\mathscr S_{\sf K} \otimes k$ of the integral model $\mathscr S_{\sf K}=\mathscr S_{{\sf K}}(G,X)$. 
\subsection{Basic definitions}\label{tensor} 
For any $b \in G(K)$, let $J_b$ be the functor on $\Q_p$-algebras defined by  
 \[J_b(R):=\{ g\in G(R \otimes _{\Q_p}K) \ : \ gb\sigma(g)^{-1}=b\}\]
for any $\Q_p$-algebra $R$. 
Then $J_b$ defines a smooth affine group scheme over $\Q_p$ by \cite[Prop. 1.12]{RZ}. 
Up to isomorphism, $J_b$ depends only on the $\sigma$-conjugacy class of $b$. 

 An element  $b \in G(K)$  is called  \emph{basic} if its slope cocharacter  
  defined by Kottwitz in \cite[$\S$4]{Kottwitz85}   
   factors through the center of $G_K$.  
   Again, this property depends only on the $\sigma$-conjugacy class of $b$.  
   An element $b$ is basic if and only if the $\Q_p$-group $J_b$ is an inner form of $G$. 

Now let $x\in S_{\sf K}(k)$. 
We write $\mathcal A_x$ for the fiber of $\mathcal A$ at $x$. 
As in \cite[3.4.2]{Kisin1}, the point $x$  gives rise to a triple $(\mathcal A_x, \lambda_x, \eta_{{\sf K}',x})$.    
Passing to the limit over ${\sf K}'^{p}$, $\eta_{{\sf K}'}$ may be promoted to a ${\sf K}^p$-orbit of an isomorphism $\Lambda_{\A_f^p} \xrightarrow{\sim} \widehat{V}^p(\mathcal A_x)$ as in \cite[3.2.4]{Kisin1}. 
We write $\eta_{{\sf K},x}$ for this orbit.  
   Furthermore, there are \emph{\'{e}tale tensors} 
 \[t_{\alpha,x}^p \in 
 \widehat{V}^p(\mathcal A_x)^{\otimes}, \]
and $\eta_{{\sf K},x}$ takes $(s_{\alpha})$ to $(t^p_{\alpha, x})$.  

Let $\mathcal A_x[p^{\infty}]$ be the $p$-divisible group of $\mathcal A_x$.  
We will write $\D(\mathcal A_x)$ for the contravariant \dieu module of  $\mathcal A_x[p^{\infty}]$.  
%Further let $\tilde{x}
%\in \mathscr S_{{\sf K}}(W)$ be a lifting of $x$. 
We have the Frobenius morphism 
\[{\sf F} : \sigma^{*}\D(\mathcal A_x) \to \D(\mathcal A_x)\]
where $\sigma$ denotes the absolute Frobenius on $W$. 
%The Hodge filtration on $\D(\mathcal A_x)$ induces a natural filtration 
%$\Fil^{\bullet}(\D(\mathcal A_x)^{\otimes}))$ on $\D(\mathcal A_x)^{\otimes}$ 
This $\sf F$ induces an isomorphism of isocrystals 
\[ {\sf F} : \sigma^{*}\D(\mathcal A_x)^{\otimes}[1/p] \xrightarrow{\sim}
\D(\mathcal A_x)^{\otimes}[1/p].\]
By \cite[Cor. 1.4.3]{Kisin1}, there are \emph
{crystalline tensors} 
\[t_{\alpha, x}\in \D(\mathcal A_x)^{\otimes}\]
that are fixed by the action of $\sf F$.

We write  $\D(\mathcal A_x)(k)$ for the reduction $\D(\mathcal A_x)\otimes_W k$. 
%By the construction of the de Rham tensors in the proof of \cite[Cor. 2.3.9]{Kisin1}, 
Then the canonical isomorphism 
\begin{align*}
\D(\mathcal A_x)(k) 
\xrightarrow{\sim}
 H^1_{{\rm dR}}(\mathcal A_x/k)
\end{align*}
sends $t_{\alpha, x}\otimes_W k$ to $t_{\alpha, {\rm dR},x}$. 

By \cite[1.4.1]{Kisin2}, there is an 
isomorphism  of $W$-modules 
 \begin{align}\label{isom}
 \Lambda_W^{*} \xrightarrow{\sim} \D(\mathcal A_x)
 \end{align}
 taking $s_{\alpha}$ to $t_{\alpha, x}$. 
 This allows us to identify the group $G_W$ with the subgroup of   
 $\GL(\D(\mathcal A_x))$ defined by $(t_{\alpha, x})$. 
  After choosing such an isomorphism, 
 the Frobenius on $\D(\mathcal A_x)$ has the form ${\sf F}=b\circ \sigma$ for some $b \in G(K)$. 
 The element $b$ is independent of choices, up to $G(W)$-$\sigma$ conjugation.

Consider the Hodge filtration 
\[\Fil^1\D(\mathcal A_x)(k) \subset \D(\mathcal A_x)(k)\cong H^1_{{\rm dR}}(\mathcal A_x/k).\]
By \cite[Cor. 1.4.3 (4)]{Kisin1}, this filtration is given by a $G_{k}$-valued cocharacter. 
By \cite[Lem.~1.1.9]{Kisin1} and the argument in its proof, 
any lift to a cocharacter 
\begin{align}\label{cochar} \mu_x : \mathbb G_{m W} \to G_W \end{align}  
is $G(\overline{K})$-conjugate to $\mu_{h}^{-1}$ 
after we fix an isomorphism $\C\xrightarrow{\sim}\overline{K}$. 
Moreover, the $G(W)$-conjugacy class of $\mu_x$ is independent of the choices of an isomorphism (\ref{isom}). 
By \cite[1.1.12]{Kisin2}, any such cocharacter satisfies 
\[b\in G(W)\mu_x^{\sigma}(p)G(W).\] 
 \subsection{Uniformization map}\label{sunif}
For each $x \in S_{{\sf K}}(k)$, we set 
 \[X_{\mu_x^{\sigma}}(b):=
 \{ g\in G(K)/G(W) : g^{-1}b\sigma (g) \in G(W)\mu_x^{\sigma}(p)G(W)\}
 . \]
If  $g_p \in G(K)$ is a representative of a point in $  X_{\mu_x^{\sigma}}(b)$,  then the submodule $g_p \cdot \D(\mathcal A_x)\subset \D(\mathcal A_x)[1/p]$ is stable under the Frobenius morphism ${\sf F}[1/p]$ and has a structure of Dieudonn\'e module. 
Hence $g_p\cdot \D(\mathcal A_x)$ with  this action  corresponds to a $p$-divisible group $\mathscr G_{g_px}$, and it is naturally equipped with a  
 quasi-isogeny $\mathcal A_x[p^{\infty}] \to \mathscr G_{g_px}$ of $p$-divisible groups corresponding to the natural isomorphism 
 \[
 g_p \cdot \D(\mathcal A_x)[1/p] \xrightarrow{\sim} \D(\mathcal A_x)[1/p]
 \]
  induced by the embedding $g_p \cdot \D(\mathcal A_x)\subset \D(\mathcal A_x)[1/p]$. 
   Further we have    
  \[t_{\alpha, x}=g_p(t_{\alpha,x}) \in 
  (g_p\D(\mathcal A_x))^{\otimes}
  =\D(\mathcal A_x)^{\otimes}.\]
Then there exist an abelian variety $A_{g_px}$ and a quasi-isogeny $\theta : \mathcal A_x \to  A_{g_px}$ corresponding to $\mathcal A_x[p^{\infty}] \to \mathscr G_{g_px}$. 
This abelian variety $A_{g_px}$ has a canonical ${\sf K}'^p$-level structure 
\[ \eta_{g_px}:=\theta_*
\circ 
\eta_{{\sf K}',x}:
  [\Lambda_{\A_f^p} 
 \xrightarrow{\sim} \widehat{V}^p(\mathcal A_{x})
 \xrightarrow{\sim} \widehat{V}^p(A_{g_px})].\]
 Since $G\subset \GSp$, the $\Z_{(p)}^{\times}$-polarization $\lambda_x$ induces a one $\lambda_{g_px}$ on  $A_{g_px}$. 
 Thus we obtain a map  
 \[X_{\mu_x^{\sigma}}(b)
   \to 
   \mathscr S_{{\sf K}'}(\GSp, S^{\pm})(k) :  g \mapsto  y=[(A_{g_p x}, \lambda_{g_p x}, \eta_{g_p x})].\]
Note that if we write  
 ${\mathcal A}'_y$ for the  fiber of   
 $\mathcal A'$ at $y$, there is a canonical prime-to-$p$ quasi-isogeny 
 $({\mathcal A}'_y, \lambda_y, \eta_{y})  
 \to 
 (A_{g_px}, \lambda_{g_px}, \eta_{g_px})
 $.
 
 By \cite[1.4.4]{Kisin2}, 
 there is a unique lifting of this map to a map
 \begin{align}\label{map} \iota_x : X_{\mu_x^{\sigma}}(b) \to S_{\sf K}(k)\end{align}
 such that its crystalline tensors satisfy equalities $t_{\alpha,  x}=t_{\alpha,  \iota_x(g_p)}$ 
 under the identification 
 \[\D(\mathcal  A_{x})^{\otimes}=
 \D(A_{g_px})^{\otimes}\xrightarrow{\sim}\D(\mathcal A'_{y})^{\otimes} = \D(\mathcal A_{\iota_x(g_p)})^{\otimes}.\] 

Now we fix a point 
\[x \in S_{{\sf K}_p}(k)
 :=\lim_{\substack{\longleftarrow 
 \\ {\sf K}^p}}
 S_{{\sf K}}(k).\] 
 Then the map (\ref{map}) extends to 
 a $G(\A_f^p)$-equivariant map
 \begin{align*}
 \iota_x : X_{\mu_x^{\sigma}}(b) \times G(\A_f^p) \to  S_{{\sf K}_p}(k). 
 \end{align*}
 By taking the quotient by each ${\sf K}^p$, one gets  
a map 
\begin{align*}
\iota_{{\sf K},x}: X_{\mu_x^{\sigma}}(b) \times G(\A_f^p) /{\sf K}^p\to S_{{\sf K}}(k).
 \end{align*}
Let $\Aut_{\Q}(\mathcal A_x)$ denote the $\Q$-group whose points in a $\Q$-algebra $R$ are 
\[\Aut _{\Q}(\mathcal A_x)(R)= (\End^0(\mathcal A_x)\otimes R)^{\times},\]
where $\End^0(\mathcal A_x)$ denotes  the endomorphism algebra of $\mathcal A_x$ viewed as an abelian variety up to isogeny. 
 We write $I \subset \Aut_{\Q}(\mathcal A_x)$ for the subgroup whose points consist of the elements fixing the tensors $t_{\alpha, x}$ and $t_{\alpha, \ell, x}$ for all $\ell \neq p$. 
 For each $\ell \neq p$ we have a  morphism $I_{\Q_{\ell}} \to G_{\Q_{\ell}}$, 
 which is canonical up to  conjugation  by  elements in  the image of ${\sf K}^p \to G(\Q_{\ell})$. 
 We also have a map $I_{\Q_p}  \to J_b$ which sends each element  $j$ in $I_{\Q_p}$ to the automorphism  $\D(j^{-1})$ of $\D(\mathcal A_x)[1/p]$. 
\begin{prop}\label{kis}\cite[Prop. 2.1.5]{Kisin2} 
The map $\iota_{{\sf K}, x}$  induces an  injective map
\begin{align*}
\vartheta_{{\sf K},x} :
I(\Q) \backslash X_{\mu_x^{\sigma}}(b) \times G(\A_f^p)/{\sf K}^p \to  S_{\sf K}(k).
 \end{align*}  
\end{prop}
For an inclusion  ${\sf K}^p_1\subset 
{\sf K}^p$ of open compact subgroups of $G(\A_f^p)$,  we set ${\sf K}_{1}:={\sf K}_p{{\sf K}^p_1}$.
 One has a projection 
\[I(\Q) \backslash X_{\mu_x^{\sigma}}(b) \times G(\A_f^p)/{\sf K}_1^p 
 \to 
 I(\Q) \backslash X_{\mu_x^{\sigma}}(b) \times G(\A_f^p)/{\sf K}^p.\]
 Then the systems of maps 
$\{\vartheta _{{\sf K}^p}\}_{{\sf K}^p}$ is compatible with the projection. 
Further, 
we take an element $h=h^p \in G(\A_f^p)$. 
The map $
g^p \cdot {\sf K}^p \mapsto  
g^ph \cdot h^{-1}{\sf K}^ph$ for $g^p \in G(\A_f^p)$ 
induces a right action  
\begin{align*}
h : 
I(\Q) \backslash X_{\mu_x^{\sigma}}(b) \times G(\A_f^p)/{\sf K}^p 
\to 
I(\Q) \backslash X_{\mu_x^{\sigma}}(b) \times G(\A_f^p)/h^{-1}{\sf K}^ph.
\end{align*} 
Then for each ${\sf K}^p$ 
 one has  a commutative diagram 
  \begin{align*}
  \begin{CD}
 I(\Q) \backslash X_{\mu_x^{\sigma}}(b)\times G(\A_f^p)/{\sf K}^p  @>{h}>>
   I(\Q) \backslash X_{\mu_x^{\sigma}}(b)\times G(\A_f^p)/h^{-1}{\sf K}^ph 
  \\
  @V{\vartheta_{{\sf K}, x}}VV 
  @V{\vartheta_{h^{-1}{\sf K}h,x}}VV 
  \\
    S_{\sf K}(k)
 @>{h}>> 
   S_{h^{-1}{\sf K}h}(k).
   \end{CD}
  \end{align*}
  
  We set  
\[  U^p:=I(\A_f^p) \cap {\sf K}^p, \ 
  I(\Z_p):=
 I(\Q_p)\cap G(W) \subset G(K).\]
 Then the natural maps   
$I(\Q_p) \to J_b(\Q_p) \to  X_{\mu_x^{\sigma}}(b)$  
 induce embeddings
 \begin{align}\label{IQ/IZ}
 I(\Q_p)/I(\Z_p) \embed J_p(\Q_p)/(J_b(\Q_p) \cap G(W))
\hookrightarrow X_{\mu_x^{\sigma}}(b).
\end{align} 
  
 Note that $I(\Q)$ is a subgroup of  the group   
 $\{g \in \End^0(\mathcal A_x)^{\times}   \vert \ g^*g\in  \Q^{\times}\}$
 where  $g \mapsto g^*$ is  the Rosati involution on $\End^0(\mathcal A_x)$ induced by the polarization on $\mathcal A_x$.  
 Since this involution is positive, every arithmetic subgroup of this group is finite. 
 Therefore $I(\Q)$ also has only finite arithmetic subgroups. 
 Hence, by   \cite[Proposition 1.4]{Gross},   
 the group $I(\Q)$ is a discrete subgroup of $I(\A_f)$ and the quotient 
 $I(\Q)\backslash I(\A_f)$ is compact. 
 Thus, the double coset  \[I(\Q) \backslash  I(\A_f)/(I(\Z_p) \cdot  U^p)\]
   is finite. 
 \begin{prop}[{\cite[Prop. 2.1.5]{Kisin2}}]\label{Z_x}
 The embedding  $\vartheta_{{\sf K},x}$ in Proposition \ref{kis} induces an  embedding 
 \begin{align}\label{unif2}I(\Q) \backslash I(\A_f)/(I(\Z_p) \cdot U^p) \hookrightarrow  S_{\sf K}(k).
 \end{align}
 \end{prop}
 This is compatible with the  projection and the right $I(\A_f^p)$-action as above. 
 
 We write $\mathcal Z(x)=\mathcal Z_{{\sf K}}(x)\subset S_{\sf K}(k)$ for the image of this map, and we also write $\mathcal Z(x)$ for the corresponding finite reduced closed subscheme of $S_{\sf K}$. 
 % The right action by $h \in I(\A_f^p)$ induces a bijection 
 % \[h : \mathcal Z_{{\sf K}}(x) \xrightarrow{\sim} 
%  \mathcal Z_{h^{-1}{\sf K}h}(x), \  i_{{\sf K},x}(y) \mapsto h(i_{{\sf K},x}(y))=i_{h^{-1}{\sf K}h,x}(y\cdot h), \]
%  where  $y\in I(\Q)\backslash I(\A_f)/(I(\Z_p) \cdot U^p)$. 
 \subsection{The basic Newton stratum and central leaves}\label{Newton}
%We write $G(K)/^{\sigma}G(K)$ for  the set of $\sigma$-conjugacy classes. 
 For a point $x \in S_{{\sf K}}(k)$, let  $b \in G(K)$ be the element %$\sigma$-conjugate  class
  associated to $x$ defined in $\S$\ref{tensor}. 
 %This assignment induces   
%  \emph{the Newton map}
%\[\mathcal N : S_{{\sf K}}(k) \to G(K) /^{\sigma}G(K), \ x \mapsto [b].\] 
 We write $\mathcal N_b \subset S_{{\sf K}}$
 %=\mathcal N_{[b]}:=\mathcal N^{-1}\{[b]\}$
  for the  
 \emph{Newton stratum} determined by   $b$.  
 By definition, a point $y$ is lying on  $\mathcal N_b(k)$ if and only if  there is a quasi-isogeny $\mathcal A_x[p^{\infty}] \to 
 \mathcal A_y[p^{\infty}]$ of $p$-divisible groups whose corresponding morphism of  \dieu isocrystals $\D(\mathcal A_y)[1/p] \to \D(\mathcal A_x)[1/p]$ maps $t_{\alpha,y}$ to $t_{\alpha,x}$. 
 Note that $x \in \mathcal N_b(k)$.  
 
There is a unique Newton stratum $\mathcal N_b$ such that $b$ is basic.   
 This stratum  $\mathcal N_b$ is called  \emph{the basic Newton stratum} of $S_{{\sf K}}$. 
This is a non-empty closed subscheme of $S_{{\sf K}}$
 (see \cite{Lee} and \cite[$\S$7.2.7]{XZ}). 
\begin{thm}
$(${\cite[Thm. 3.3.2]{HP} {\rm and} \cite[Cor. 7.2.16]{XZ}}$)$\label{basic}
Assume that $x$ is lying on the basic Newton stratum $\mathcal N_b$. 
 Then, the embedding $\vartheta_{{\sf K},x}$ in Proposition \ref{kis} induces a bijection
 \[I(\Q) \backslash X_{\mu_x^{\sigma}}(b) \times G(\A_f^p)/{\sf K}^p \xrightarrow{\sim} \mathcal N_b(k).\]   
Moreover,  the group scheme $I$ is an inner form of $G$, and  there are natural identifications
\[
I_{\Q_{\ell}}=
\begin{cases}J_b & {\rm if} \ \ell=p ;
\\
G_{\Q_{\ell}} &  {\rm  if} \ 
\ell \neq p.
\end{cases} \]
\end{thm}
 We remark that Howard and Pappas constructed a  Hodge type Rapoport-Zink formal scheme 
${\rm RZ}_G$ over ${\rm Spf}(W)$, and showed that there is an isomorphism of formal schemes 
\[I(\Q) 
\backslash {\rm RZ}_G \times G(\A_f) 
/{\sf K}^p 
\xrightarrow{\sim} 
(\widehat{\mathscr S_{{\sf K}}\otimes W})_{/\mathcal N_b}\]
where 
the right hand side  
is the completion of 
$\mathscr S_{{\sf K}} \otimes W$ 
along the basic Newton stratum 
$\mathcal N_b$. 
The above bijection is obtained from taking the sets of $k$-valued points of this isomorphism. 

Any two points $x,y \in S_{\sf K}(k)$  are said to be in the same \emph{central leaf} if there exists an isomorphism of Dieudonn\'e modules
$\D(\mathcal  A_y) \to \D(\mathcal A_x)$ mapping $t_{\alpha,y}$ to $t_{\alpha,x}$. 
We write $\mathcal C(x)=\mathcal C_{{\sf K}}(x)$ for the central leaf passing through $x$.
\begin{cor}\label{c=z}
Assume that $x$ is lying on the basic Newton stratum $\mathcal N_b$. 
  Then we have  $\mathcal Z(x)=\mathcal C(x)$. 
In particular, $\mathcal C(x)$ is a finite set.
\end{cor}
We remark that the finiteness  also follows from a dimension formula for  central leaves given by C.~Zhang \cite[Thm. 2 (3)] {Zhang2}.
\begin{proof}
For any $g_p \in I(\Q_p)$, the map  \[ \D(\mathcal A_x) \to g_p\cdot \D(\mathcal A_x) \xrightarrow{\sim}\D(\mathcal A_{\iota_x(g_p)}): u \mapsto g_pu\] 
 is an isomorphism of Dieudonn\'e modules taking $t_{\alpha,x}$ to $g_pt_{\alpha,x}=t_{\alpha,x}=t_{\alpha, \iota_x(g_p)}$. 
Hence the set  $\mathcal Z(x)=\iota_x(I(\Q_p)/I(\Z_p) \times I(\A_f^p))$ is contained in  $\mathcal C(x)$.

Conversely, if  $y \in \mathcal C(x)$  then one has an  isomorphism of Dieudonn\'e modules 
\[h : \D(\mathcal A_y)
 \xrightarrow{\sim}
  \D(\mathcal A_x)\]
taking $t_{\alpha,y}$ to $t_{\alpha, x}$. 
Further, Theorem \ref{basic} implies that  $y=\iota_x((g_p \cdot G(W), g^p))$ for some $(g_p \cdot G(W), g^p) \in X_{\mu_x^{\sigma}}(b)\times G(\A_f^p)$ because $\mathcal C(x)$ is contained in the basic Newton stratum $\mathcal N_b(k)$. 
In particular one has an identification $\D(\mathcal A_y)
\xrightarrow{\sim} g_p\cdot \D(\mathcal A_x)$ such that  $t_{\alpha,y}=g_pt_{\alpha, x}=t_{\alpha,x}$.  
 Under this identification, the base extension   
  $h\otimes_{\Z_p} \Q_p$ defines an element in  $J_b(\Q_p)$, and the composition $h \circ (g_p\vert_{\D(\mathcal A_x)})$  defines an element in $G(W)$. 
  Hence   
 $g_p$ and $h^{-1}\otimes \Q_p$ define the same class
 in $X_{\mu_x^{\sigma}}(b)$,  and hence in $J_b(\Q_p)/(J_b(\Q_p) \cap G(W)).$ 
 Therefore, by Theorem \ref{basic}, the element $(g_p \cdot G(W), g^p)$ comes from $I(\Q_p)/I(\Z_p) \times I(\A_f^p)$ through the embedding (\ref{IQ/IZ}).   Thus, the point $y$ lies in $ \mathcal Z(x)$. 
\end{proof}
\section{Automorphic forms$\pmod p$ on Shimura varieties}\label{ssaut}
First we recall definitions and some   properties of  automorphic forms on  Shimura varieties. 
References are  \cite{GK} for  Hodge type and  \cite{Lan2} for PEL type. 
Next we see that  automorphic forms on a finite subset of the special fiber $S_{{\sf K}}(k)$ can be regarded as vector valued functions on the scheme associated to the principal bundle over this subset.
Finally, we construct a variant of the uniformization map, which will allow us to  
identify the automorphic forms on $\mathcal Z(x)$ with  algebraic modular forms.  
\subsection{Principal bundles and automorphic bundles}\label{autbdl}
Since $G_{\Z_p}$ is quasi-split, 
we may assume that the representative $\mu$ of $[\mu_h]$ extends to a cocharacter 
$\mu : \mathbb G_{m,W} \to G_W$. 
Define 
\[\mu_{\varphi} : \mathbb G_{m, W} 
\to 
\GSp(\Lambda_{\Z_{(p)}}, \psi)_W\] 
by $\mu_{\varphi} :=\varphi \circ \mu$ where $\varphi$ is the embedding defined in (\ref{iota}). 
Then the cocharacter $\mu_{\varphi}$ defines a decomposition $\Lambda_{\Z_{(p)}}\otimes W=\Lambda_0 \oplus \Lambda_{-1}$ over $W$  where  
 $\mathbb G_m$ acts through $\mu_{\varphi}$ by $z \mapsto z^{-i}$ on $\Lambda_i$ for $i=0, -1$.  

 Let $P_W\subset G_W$ be the parabolic subgroup whose Lie algebra consists of the non-negative weight spaces for $\mu$.  
  Then we have an identification $P_W=\{g \in G_W \vert \ g(\Lambda_0)=\Lambda_0\}$. 
  
  Let $L_W$ be the centralizer of $\mu$ in $G_W$. 
  Then $L_W$ is a Levi subgroup of $P_W$, i.e. a reductive closed subgroup defined over $W$ such that the canonical homomorphism $L_W \to P_W/R_u(P_W)$ is an isomorphism 
  where $R_u(\cdot)$ denotes the unipotent radical  (see \cite[A.6]{VW}). 
  %In particular, $L$ is flat over $W$.
  
 Put $H^1_{\rm{dR}}(\mathcal A/\mathscr S_{\sf K}) 
 :=\varphi ^{*}H^1_{\rm{dR}}(\mathcal A'/\mathscr S_{{{\sf K}'}}(\GSp, S^{\pm}))$ and \[\Fil^1(H^1_{\rm{dR}}(\mathcal A/\mathscr S_{\sf K}))=\Lie^t(\mathcal A):=\varphi^{*}\Lie^t(\mathcal A ').\] 
 We set a sheaf $\mathcal P$ on $\mathscr S_{{\sf K}}\otimes W$ of isomorphisms of coherent sheaves  preserving the  filtrations and the tensors 
 \begin{align*}
 \mathcal P=\mathcal P_{{\sf K}}:=\Isom_{\mathcal O_{\mathscr S_{\sf K}\otimes W}}( [\Lambda_{\Z_p}^{*}& \otimes \mathcal O_{\mathscr S_{\sf K}\otimes W},  \Lambda_0^{*}\otimes \mathcal O_{\mathscr S_{\sf K}\otimes W}, (s_{\alpha})],\\
 & [H^1_{\rm{dR}}(\mathcal A/\mathscr S_{\sf K}\otimes W), \Fil^1 H^1_{{\text dR}}, (t_{\alpha,\rm{dR}})] 
 ).\end{align*}
 By \cite[Prop.~4.3.9]{Madapusi}, this  $\mathcal P$ is a $P_W$-torsor. 
 We define an $L_W$-torsor $\mathcal L=\mathcal L_{{\sf K}}$ as the quotient $\mathcal P / R_u(P_W)$. 
 
We write $L$ (resp.~$P$) for the reduction $L_W \otimes _W k$ (resp. ~$P_W \otimes_Wk$). 
 Let $V$ be a rational  representation of  $L$ over $k$. 
 The \emph{automorphic bundle} 
 $\mathscr V(V)$  of weight $V$ is the vector  bundle on $S_{\sf K}$ defined by
 \[\mathscr V(V):=\mathcal L \vert_{S_{\sf K}} \times^{L} V=(\mathcal L \vert _{S_{\sf K}} \times V)/L\]
 where $L$ acts on $\mathcal L \vert _{S_{\sf K}}\times V$ via $(\phi, v)g=(\phi g, g^{-1}v)$. 
% By , for an open subset $U\subset \mathscr S_{\sf K}$  we have 
% \begin{align}\label{ab}
% \mathscr V(\gamma)(U)=
% \{f \in {\rm Map}_{U}(\mathbb E_L\vert_{U}, \ \mathbb V(V_{\gamma})\vert _{U}) \ \vert \  
% f(xg)=g^{-1}f(x), g\in 
% L
 %\}
 %\end{align}
%where $\mathbb E_L \to \mathbb E_L/L=\mathscr S_{\sf K}$ is the  principal bundle scheme associated to the sheaf $\mathcal L$, and $\mathbb{V}(V_{\gamma})= \Spec (\Sym^{\bullet}(V_{\gamma}\otimes \mathcal O_{\mathscr S_{\sf K}})) \to \mathscr S_{\sf K}$ is the vector bundle scheme associated to $V_{\gamma}$. 

The rational irreducible representations of $L$ are classified by   means of maximal tori and highest weights.  
 Fix a maximal torus $T$  of  $L$. 
 Let $X^*(T)$ (resp. $X_*(T)$) be the group of characters (resp. cocharacters) of $T$ and let 
$\langle \ , \ \rangle : X^*(T) \times X_{*}(T) \to \Z$ 
 be the perfect pairing. 
Write $\Phi:=\Phi(L,T)$ for the root system of  $L$ with respect to $T$.  
We choose a positive system  $\Phi^{+}\subset \Phi$ and set 
\[X_{+}^*=X^{*}_{+, L}(T):=\{\xi  \in X^{*}(T) \ \vert \ \langle \xi, \alpha^{\vee} \rangle \geq 0 \ \rm{for} \ \rm{all} \ \alpha \in \Phi^{+} \}.\] 
The elements of $X_{+}^*$ are called the \emph{dominant weights} of $T$ (with respect to $\Phi^{+}$). 
 
Let  $B$ be the  Borel subgroup of $L$ which is  generated by  the root groups of  negative roots. 
(This is the same convention as in \cite[II.1.8]{Jantzen}.)  
Then $B$  contains  $T$. 

For  each dominant weight $\xi \in X_{+}^*$, let $k_{\xi}$ be $k$ considered as a one dimensional rational  representation of $B$ via $\xi$. 
 Then $B$ acts freely on $L \times k_{\xi}$ via $(x,y)g=(xg, g^{-1}y)$.  
 We define a sheaf $\mathscr L(\xi)$ on the flag variety $L/B$ by 
 \[\mathscr L(\xi):=L \times ^{B}k_{\xi}=(L\times k_{\xi})/B.\] 
 Then the space 
  $H^0(\xi):=H^0(L/B, \mathscr L(\xi))$ is of  finite dimension and regarded as a rational  representation of $L$.  
 Furthermore, the maximal semisimple  subrepresentation  $V_{\xi}:={\rm soc}_{L}H^0({\xi})$ is the irreducible representation  with highest weight $\xi$. 
Then $V_{\xi}$ 
  with
   $\xi \in X_{+}^*$ are a system of representatives for the isomorphism classes of  rational irreducible   representations of $L$ over $k$   
   (cf. \cite[II.~Cor.~2.6]{Jantzen}).  
   We write $\mathscr V(\xi)$ for the automorphic bundle $\mathscr V(V_{\xi})$. 
   
 We remark that this notation is different from the one in \cite[$\S$4.1.9]{GK}, where  modules $V_{\eta}$ were defined as $H^0(\eta)$ above but  defined over the ring of integers of a finite  extension of $E_{\mathfrak p}$. 
In particular, the reductions  $V_{\eta} \otimes k$ were not  necessarily  irreducible as a representation of $L$. 
\subsection{Extensions to   toroidal compactifications}\label{ext} 
 Let ${\Sigma}'$ be a finite, admissible rpcd (rational, polyhedral cone decomposition) 
for the Siegel Shimura datum $({\rm{GSp}}, S^{\pm}, {{\sf K}'})$. 
By \cite{faltings-chai}, there exits a toroidal compactification $\mathscr S_{{\sf K}'}(\GSp, S^{\pm})^{\Sigma'}$ of $\mathscr S_{{\sf K}'}(\GSp, S^{\pm})$. 
Further, by \cite[Thm.~2.15 (1) and (2)]{Lan2}, 
the universal family $\mathcal A' \to \mathscr S_{{\sf K}'}(\GSp, S^{\pm})$ extends to a proper and log smooth morphism 
$\beta : \overline{\mathcal A'} \to 
\mathscr S_{{\sf K}'}(\GSp, S^{\pm})^{\Sigma'}$. 
By \cite[Thm.~2.15 (c) and (d)]{Lan2},  this morphism  induces  the log de Rham complex $\overline{\Omega}^{\bullet}_{\overline{\mathcal A'} / 
\mathscr S_{{\sf K}'}(\GSp, S^{\pm})^{\Sigma'}}$ and  
the log de Rham cohomology 
\[H^1_{\text{log-dR}}
(\overline{\mathcal A'}/\mathscr S_{{{\sf K}'}}(\GSp, S^{\pm})^{{\Sigma}'}):=R^1 \beta_* (\overline{\Omega}^{\bullet}_{\overline{\mathcal A'} / 
\mathscr S_{{\sf K}'}(\GSp, S^{\pm})^{\Sigma'}}),\]
which is a locally free   extension of $H^1_{\rm{dR}}
(\mathcal A'/\mathscr S_{{{\sf K}'}}(\GSp, S^{\pm}))$. 

Let $\Sigma$ be a refinement of the finite, admissible rpcd obtained from the pull-back $\varphi^*{\Sigma'}$ for $(G, X, {\sf K})$. 
By \cite[Thm.~4.1.5]{Madapusi}, there exist toroidal compactifications 
$\mathscr S_{\sf K}^{\Sigma}$ and  
$\mathscr S_{{\sf K}}^{\varphi^{*}\Sigma'}$ 
 of $\mathscr S_{\sf K}=\mathscr S_{\sf K}(G,X)$. 
 Then the map $\varphi$ extends to a map 
 $\varphi^{\Sigma/\Sigma'} : 
 \mathscr S_{\sf K}^{\Sigma} 
 \to 
 \mathscr S_{{\sf K}'}(\GSp, S^{\pm})^{\Sigma'}$. 
 We put 
  \begin{align*}  H^1_{\text{log-dR}}(\overline{\mathcal A}/\mathscr S_{\sf K}^{\Sigma}) :&=
 (\varphi^{\Sigma/\Sigma'})^*(H^1_{\text{log-dR}}(\overline{\mathcal A'}/\mathscr S_{{{\sf K}'}}(\GSp, S^{\pm})^{{\Sigma}'})), 
 \\
  \Fil^1H^1_{\text{log-dR}}(\overline{\mathcal A}/\mathscr S_{\sf K}^{\Sigma})  &=\Lie^t(\overline{\mathcal A}) := (\varphi^{\Sigma'/\Sigma})^{*}\Lie^t \overline{\mathcal A'}.\end{align*} 
 These give locally free extensions of $H^1_{\rm{dR}}(\mathcal A/\mathscr S_{\sf K})$ and $\Lie^t(\mathcal A)$ respectively.

By \cite[Prop.~4.3.7 (1)]{Madapusi}, 
        the tensors $(t_{\alpha,\rm{dR}})$ 
        extend uniquely to 
        $H^1_{\text{log-dR}}(\overline{\mathcal A}/\mathscr S_{\sf K}^{\Sigma})$. 
        By repeating the definition of 
        $\mathcal P$
         with 
         $\mathscr S_{\sf K}^{\Sigma}\otimes W$, 
         the extended tensors,  and $H^1_{\text{log-dR}}(\overline{\mathcal A}/\mathscr S_{\sf K}^{\Sigma})$, 
        we get an extension to $P_W$-torsor  $\mathcal P^{\Sigma}$ on $\mathscr S_{\sf K}^{\Sigma} \otimes W$ (\cite[Proposition 4.3.9]{Madapusi}). 
        Again we set a $L_W$-torsor \[\mathcal L^{\Sigma}:=
        \mathcal P^{\Sigma}/R_u(P)\]
         on $\mathscr S_{\sf K}^{\Sigma}\otimes W$. 
        
For each  $\xi \in X_{+}^*$, we define  a vector bundle $\mathscr V^{\rm{can}}(\xi)$ on the special fiber $S_{{\sf K}}^{\Sigma}$ by 
         \[\mathscr V^{\rm{can}}(\xi):=\mathcal L^{\Sigma} \vert_{S^{\Sigma}_{\sf K}} \times^{L}V_{\xi}. \]   
         Let $D=D^{\Sigma}_{\sf K}$ be the boundary divisor of $\mathscr S^{\Sigma}_{\sf K}$ relative to $\mathscr S_{\sf K}$. 
         Set $\mathscr V^{\rm{sub}}(\xi):=\mathscr V^{\rm{can}}(\xi)(-(D\otimes k ))$. 
         \begin{defn}
         For any $\xi \in X_{+}^*$, 
        we call $H^0(S_{{\sf K}}^{\Sigma}, \mathscr V^{{\rm can}}(\xi))$ (resp.~$H^0(S_{{\sf K}}^{\Sigma}, \mathscr V^{{\rm sub}}(\xi))$) the space of automorphic forms (resp.~cusp forms) on $S_{{\sf K}}^{\Sigma}$ of weight $\xi$. 
         \end{defn}

\subsection{Hecke operators and systems of eigenvalues}
Let  $v \neq p$ be  an unramified finite place for $G$ and ${\sf K}_v$ be a hyperspecial subgroup of $G(\Q_v)$. 
We write $\mathcal H_v=\mathcal H_v(G_v, {\sf K}_v; \Z_p)$ for the unramified Hecke algebra of $G$ at $v$  with $\Z_p$-coefficients,
  normalized by the unique Haar measure with ${\rm vol}({\sf K}_v)=1$. 
We define the unramified, global Hecke algebra $\mathcal H$ by the restricted tensor product 
\[\mathcal H=\mathcal H(G):={\bigotimes_{v  : {\rm unr}, v \neq p}}' \mathcal H_v.\]
Let $M$ be a finite dimensional $k$-vector space which is also an $\mathcal H$-module. 
We say that 
 a  system of Hecke eigenvalues $(b_T)_{T \in \mathcal H}$ appears    in $M$ 
 if there exists $f 
\in M$ such that $T  f=b_T  f$ for all $T \in \mathcal H$. 

Let $g\in G(\Q_v)$ for an unramified finite place $v \neq p$.  
The algebra  $\mathcal H_v$ is generated by the characteristic functions of the double cosets ${\sf K}_v g {\sf K}_v$, and  each characteristic function defines a linear operator $T_g$ of $H^0(S_{{\sf K}}, \mathscr V(\xi))$ as follows. 
Let $\pi_1:=\pi_{({\sf K}\cap g{\sf K}g^{-1})/{\sf K}} $ and $\pi_2:=\pi_{(g^{-1}{\sf K}g\cap {\sf K})/{\sf K}}\circ g$,   where 
\[g : \mathscr S_{{\sf K}\cap g{\sf K}g^{-1}} \xrightarrow{\sim} \mathscr S_{g^{-1}{\sf K}g\cap {\sf K}}\]
 is the right action. 
Then 
\[\pi_j  : \mathscr S_{{\sf K}\cap g{\sf K}g^{-1}} \to \mathscr S_{\sf K}\]
 are finite \'etale projections. 
Thus we have a trace map 
\[\tr \pi_j : 
\pi_{j,*}
\mathcal O_{\mathscr S_{{\sf K}\cap g{\sf K}g^{-1}}}
\to \mathcal O_{\mathscr S_{{\sf K}}}.\]
The automorphic vector bundle $\mathscr V(\xi)$ is  a $G(\A_f^p)$-equivariant coherent sheaf on the tower $(S_{\sf K})_{{\sf K}^p}$, in the sense of \cite[$\S$4.1.10]{GK}. 
Then $\pi_1^{*}\mathscr \mathscr V(\xi)=\pi_2^{*}\mathscr V(\xi)$ and  
the Hecke operator 
\[T_g :H^0(S_{\sf K}, \mathscr V(\xi)) \to H^0(S_{\sf K}, \mathscr V(\xi))\]
is defined by $T_g:=\tr \pi_2\circ \pi_1^{*}$. 

By \cite[$\S$8.1.5]{GK}, one has  the Hecke operator on a toroidal  compactification  
\[T_g^{\Sigma} :
 H^0(S_{\sf K}^{\Sigma}, 
 \mathscr V^{\bullet}(\xi))
  \to 
  H^0(S_{\sf K}^{\Sigma},
   \mathscr V^{\bullet}(\xi))\]
   for $\bullet\in \{\rm{can, sub}\}$.  
   The definition is more complicated than that on $S_{{\sf K}}$.

 By extending linearly to $\mathcal H_v$ and to the restricted tensor product for  all $v$, one  gets  $\mathcal H$-module structures of  $H^0(S_{{\sf K}}, 
   \mathscr V(\xi))$ and $H^0(S^{\Sigma}_{{\sf K}},
    \mathscr V^{\bullet}(\xi))$. 
   \subsection{Automorphic forms on finite subschemes} 
Let $\mathcal Y$ be a finite  reduced subscheme of $S_{{\sf K}}$ over $k$. 
 We show that automorphic forms on $\mathcal Y$ can be regarded as $L$-invariant vector valued  functions on the principal bundle $\mathcal L_{\mathcal Y}$. 
We write $\pi : \mathcal L_{\mathcal Y} \to \mathcal Y$ for the scheme associated to   the sheaf $\mathcal L \vert _{\mathcal Y}$ on $\mathcal Y$. 
We also  
 write  $\mathcal L_{\mathcal Y}$ for the set  of its $k$-points, for simplicity. 
 Then the fiber 
 $\mathcal  L_y:=\pi^{-1}(y)$ at $y \in \mathcal Y(k)$ can be identified with the space $H^0(\{y\}, \mathcal L \vert_{\{y\}})$.    
Since $\mathcal Y$ is finite and reduced, the bundle $\mathcal L_{{\mathcal Y}}
 =\bigsqcup_{y \in \mathcal Y} \mathcal L_y$  is a trivial $L$-torsor on $\mathcal Y$.

 Now let $V$ be a rational  representation of $L$ over $k$.
 We write 
  ${\rm Map}(\mathcal L_{\mathcal Y}, V)$ for  the set of functions on $\mathcal L_{\mathcal Y}$ to $V$. 
  For such a function $f$ and for an element $g \in L$, 
  we define a function $f^{g}$ by the rule  $f^{g}(z)=g^{-1}f(z \cdot g^{-1})$
   for $z \in \mathcal  L_{\mathcal Y}$. 
 This assignment $f \mapsto f^{g}$ induces a right $L$-action on ${\rm Map}
 (\mathcal 
 L_{\mathcal Y}, V)$. 
  \begin{lemma}[{cf. \cite[I.5.15]{Jantzen}}]\label{fun=aut}
There is  a natural bijection  
 \begin{align*}
 {\rm Map}(\mathcal L_{\mathcal Y}, V)^{L} 
 \xrightarrow{\sim}
 H^0(\mathcal Y, \mathscr V(V)).
 \end{align*}
  \end{lemma}
  \begin{proof} 
By the definition of the bundle $\mathscr V(V)$,   one has an isomorphism  
 \begin{align*}\label{H^0}
 H^0(\mathcal Y, \mathscr V(V))
 \xrightarrow{\sim}
 \prod_{y \in \mathcal Y}
 \big (\mathcal L_y\times
 V)/  L 
 \end{align*}
 where $g \in L$ acts on $\mathcal L_y \times
 V$ via  
 $(z, v)\cdot g=(z g, g^{-1}v).$
 
 Let $f \in {\rm Map}(\mathcal L_{\mathcal Y}, V)^{L}$.  
  For each $y \in \mathcal Y$, we put  
  \[s_{f, y}:=\{(z, f(z)) : z \in \mathcal L_y\}\subset \mathcal L_y \times V.\]
   Since $f$ is $L$-invariant and $\mathcal L$ is an $L$-torsor, the set $s_{f, y}$ defines an $L$-orbit. 
 Thus we obtain an element  $s_f :=(s_{f,y})_y \in \prod_{y \in \mathcal Y}
 \big (\mathcal L_y\times
 V)/  L$.  
 
 On the other hand, take any $s=(s_y) \in \prod_{y \in \mathcal Y}
 \big ((\mathcal L_y\times
 V)/  L \big)$. 
 Then for each $z \in \mathcal L_{\mathcal Y}$ there exists a unique $v \in V$ such that the point $(z, v)$ belongs  to the $L$-orbit  $s_{\pi(z)}$, because $L$ acts freely on $\mathcal L_y\times
 V$. 
 One can see the function $f_s : \mathcal L_{\mathcal Y} \to V, \ z \mapsto v$ is $L$-invariant.
   Further, the  map $s \mapsto f_s$ is an inverse of $f \mapsto s_f$.  
\end{proof}        
 \subsection{Uniformization with sections of the principal bundle}\label{unifbdl}
Recall that we fixed an isomorphism $\Lambda_{\Z_{(p)}}^* \otimes W \cong \D(\mathcal A_x)$ in (\ref{isom}). 
We write $\mu_{x,k}$ for a  cocharacter $\mathbb G_{m k} \to G_k$ defining the Hodge filtration 
\[\Fil^1(\D(\mathcal A_x)(k))=\Lie^t(\mathcal A_x[p^{\infty}])(k) \subset \D(\mathcal A_x)(k)\]
as in $\S$\ref{tensor}. 
Let $P_x \subset G_k$ be the parabolic subgroup defined by $\mu_{x,k}$.
 Let $L_x$ be the centralizer of $\mu_{x,k}$ in $G_k$. 
 Then $L_x$ is a Levi subgroup of $P_x$ and there is a natural isomorphism 
 \[P_x/ R_u(P_x) \xrightarrow{\sim} L_x\]
  where $R_u(P_x)$ is the maximal unipotent subgroup of $P_x$. 
 One can regard $L_x$ as a subgroup of 
 \[\Aut_k({\rm gr}^{\bullet}(\D(\mathcal A_x)(k))):=\GL_k (\Fil^1 \D 
 (\mathcal A_x)(k))\times \GL_k(\D(\mathcal A_x)(k)/\Fil^1\D 
 (\mathcal A_x)(k)).\]

Let $g\in I(\Z_p)=I(\Q_p) \cap G(W)$. 
The reduction  $g \otimes_W k$ induces an automorphism of   $\D(\mathcal A_x)(k)$  preserving  the Hodge filtration. 
Hence we have a morphism from $I(\Z_p)$ to $P_x$ sending $g$ to $g \otimes _W k$. 
We write $U_p$ for the kernel  of the composition 
\begin{align}\label{pi}\varpi : I(\Z_p) \to P_x \to L_x.\end{align}
In other words,  
  the subgroup $U_p$ consists of  elements $g \in I(\Z_p)$ such that  
  $(g\otimes k) \vert_{\Fil^1}=\id$ and $g\otimes k \ (\bmod \ {\Fil^1})=\id$. 
   Then $U_p$ is an  open compact subgroup of $I(\Z_p)$. 
  We set a finite group   $I(p):=I(\Z_p)/U_p$ and regard it as a subgroup of $L_x$ via $\varpi$.

Now we fix a section   
  \begin{align*}
  \delta 
   \in \mathcal P_{x} & \cong 
  H^0(\{x\}, \mathcal P \vert _{\{x\}})
  \\
  & =
  \Isom ((\Lambda^*_k, 
 \Lambda^*_0\otimes k,
 (s_{\alpha}\otimes k)), (\D(\mathcal A_x)(k), {\rm \Fil}^1(\D(\mathcal A_x)(k)), (t_{\alpha,x}\otimes k))).
 \end{align*} 
 (We remark that we can choose an isomorphism  in (\ref{isom}) such that its reduction$\pmod p$ induces a section in $\mathcal P_x$.  
 However, that is not needed   for our purpose.)
 Then $\delta$ induces an identifications of the parabolic subgroups and 
  their Levi quotients 
 \begin{align}\label{P=P_x}P \xrightarrow{\sim}P_x, \ L \xrightarrow{\sim} L_x.
 \end{align} 

  We will construct a map
\begin{align}\label{diff}
\begin{split}
 I(\Q_p) & \rightarrow 
 \mathcal L_{{\mathcal Z}(x)}
  \end{split}
\end{align} 
by modifying the map $\iota_x$ in (\ref{map}). 
Let $g_p \in I(\Q_p)$. 
  The automorphism $g_p$ of $\D(\mathcal A_x)[1/p]$ induces an identification  
  \[g_p \vert_{\D(\mathcal A_x)} : \D(\mathcal A_x)\xrightarrow{\sim} g_p\cdot \D(\mathcal A_{x})=\D( A_{g_px})\] 
  of $W$-modules preserving the tensors. 
  Recall that the right hand side is equipped with a structure of Dieudonn\'{e} module, by restricting the Frobenius morphism ${\sf F}[1/p]$ of $\D(\mathcal A_x)[1/p]$.  
  Since  $g_p$ commutes with ${\sf F}[1/p]$, i.e.~$I(\Q_p)\subset  J_b(\Q_p)$,  
   the above identification $g_p \vert_{\D(\mathcal A_x)}$ preserves the structures  of Dieudonn\'{e} modules. 

%Recall that, in $\S$\ref{sunif}, 
% for each $g_p$  we constructed 
%a point   
%    $(A_{g_px}, \lambda_{g_px}, \eta_{g_px}) \in \mathscr S_{{\sf K}'}(\GSp, S^{\pm})(k)$ and a lift $\iota_x(g_p) \in S_K(k)$.
%Further, we have $\iota_x(g)=g^p (\iota_x(g^p))$. 
%    On the other hand, 
 %    the right action of $g^p$ on $\mathscr S_{{\sf K}'}(\GSp, S^{\pm})$ sends $(A_{g_px}, \lambda_{g_px}, \eta_{g_px})$ 
 %    to $(A_{g_px}, \lambda_{g_px}, \eta_{gx})$ where 
 %% \[
 %\eta_{gx}:=\eta_{g_px} \circ g^p : \Lambda^*_{\A_f^p}  \xrightarrow{\sim} \hat{T}^p( A_{g_px})_{\Q}.\]
 %Hence, the map  $\mathscr S_K \to \mathscr S_{{\sf K}'}(\GSp, S^{\pm})$ sends $\iota_x(g)$ to  
 %$(A_{g_px}, \lambda_{g_px}, \eta_{gx})$. 
  Recall that we write  $\mathcal A' \to \mathscr S_{{\sf K}'}(\GSp, S^{\pm})$ for the universal abelian scheme and $\mathcal A \to \mathscr S_{{\sf K}}$ for the pull-back via $\varphi$. 
 One has a canonical prime-to-$p$ quasi-isogeny $\mathcal A'_{\varphi(\iota_x(g_p))} \to A_{g_px}$, which induces an isomorphism 
   $\D(A_{g_px})\xrightarrow{\sim}
   \D(\mathcal A'_{\varphi(\iota_x(g_p))})=\D(\mathcal A_{\iota_x(g_p)})$  
   of Dieudonn\'{e} modules. 
 Under this identification, we obtain an isomorphism   
 \[ g_p \vert_{\D(\mathcal A_x)} \otimes_W k : \D(\mathcal A_x)(k)\xrightarrow{\sim} \D(\mathcal A_{\iota_x(g_p)})(k)\] preserving the filtrations and the crystalline tensors. 
  
  We define an element $\phi({g_p}) \in 
  \mathcal P_{\iota_x(g_p)}$  by the composition 
  \[
  \phi({g_p}):= 
  (g_p \vert_{\D(\mathcal A_x)}\otimes_W k) \cdot \delta  \ : \ 
   \Lambda^*_k 
    \xrightarrow{\sim} \D(\mathcal A_{\iota_x(g_p)})(k)\] 
    where $\delta$ is the fixed section. 
  By taking the  $R_u(P)$-orbit, we get an  element  \[\phi(g_p) R_u(P) \in  
  \mathcal L_{\iota_x(g_p)}.\]
 Since $\mathcal L_{\mathcal Z(x)}$ is the union $\bigsqcup_{y \in \mathcal Z(x)}\mathcal L_y$,  we obtain the desired map (\ref{diff}).    
 
For an inclusion ${\sf K}^p_1 \subset {\sf K}^p$ of open compact subgroups of $G(\A_f^p)$, we set ${\sf K}_1:={\sf K}_p{\sf K}_1^p$. 
Then the 
  projection $\pi_{{\sf K}_1/{\sf K}}$ induces a map 
 $\mathcal L_{\mathcal Z_{{\sf K}_1}(x)} 
 \to 
 \mathcal L_{\mathcal Z_{{\sf K}}(x)}$. 
 % We define a right $I(\A_f^p)$-action on $\mathcal L_{\mathcal Z(x)}$ as follows. ??? 
%Recall that we write  $\mathcal P_{{\sf K}}$ and $\mathcal P_{ h^{-1}{\sf K}h}$ for the principal bundles on $\mathscr S_{\sf K}$ and $\mathscr S_{h^{-1}{\sf K}h}$, respectively. 
% For each $h \in I(\A_f^p)$ and for each $y \in \mathcal Z_{{\sf K}}(x)$ there is a canonical identification 
% $\theta^*:\D(\mathcal A_{y}) \xrightarrow{\sim} \D(\mathcal A_{h(y)})$ where $\mathcal A_{h(y)}$ is the fiber over $h(y) \in \mathcal Z_{h^{-1}{\sf K}h}(x)$. 
% Hence there is an bijection   
% \[(\mathcal P_{{\sf K}})_y \xrightarrow{\sim} (\mathcal P_{ h^{-1}{\sf K}h})_{h(y)}, \  \phi \mapsto \theta^{*}\phi .\] 
% We define the right action 
% \[
% h:  \bigsqcup_{y \in \mathcal Z_{{\sf K}}(x)}
 % (\mathcal L_{ {\sf K}})_y 
 % \xrightarrow{\sim}
 % \bigsqcup_{z 
 % \in 
%  \mathcal Z_{h^{-1}{\sf K}h}(x)}
%  (\mathcal L_{ h^{-1}{\sf K}h})_z
%  \]
%  by taking the quotient  and the unions. 
The construction of map (\ref{diff}) is compatible with the level structures ${\sf K}^p$. 
Now we pick a point  
$x\in S_{{\sf K}_p}(k)$. 
Then we obtain an $I(\A_f^p)$-equivariant map 
 \[I(\A_f) \to \mathcal L_{\mathcal Z_{{\sf K}_p}(x)}:=\lim_{\substack{
 \longleftarrow \\
 {\sf K}^p}}\mathcal L_{\mathcal Z_{{\sf K}}(x)}.\]
 Let $U^p:=I(\A_f^p)\cap {\sf K}^p$. By taking the quotient, one has a map
  \[I(\A_f)/U^p \to 
  \mathcal L_{\mathcal Z_{{\sf K}}(x)}.\] 
\begin{prop}\label{propemb}
Let $U=U_pU^p$. 
Then the above map  induces an 
embedding 
\begin{align}\label{emb}
I(\Q)\backslash I(\A_f)/U & \rightarrow \mathcal L_{{\mathcal Z}(x)}, 
\end{align}
such that 
 the following diagram is commutative:
\[
\begin{CD}I(\Q)\backslash
I(\A_f)/U @>>>
\mathcal L_{\mathcal Z(x)} 
\\ 
@VVV  @VVV 
\\
I(\Q)\backslash 
I(\A_f)/I(\Z_p)U^p 
@>{\sim}>>
\mathcal Z(x).
\end{CD}
\]
\end{prop}
\begin{proof}
First we show  the map is well-defined. 
The group $I(\Z_p)$ acts stably on $\D(\mathcal A_x)$.  
Hence, for any $g_p \in I(\Q_p)$ and  $u \in U_p (\subset I(\Z_p))$ one has 
  \begin{align*}
  \phi({g_pu})&= 
  (g_pu \vert_{\D(\mathcal A_x)}\otimes k) \cdot \delta  \\
  &=
  (g_p \vert_{\D(\mathcal A_x)}\otimes k) \cdot (u\otimes k) \cdot  \delta \\
  &= 
  (g_p \vert_{\D(\mathcal A_x)}\otimes k) \cdot \delta \cdot \delta^{-1} \cdot (u \otimes k)  \cdot \delta 
  \\
  &= \phi(g_p) \cdot  \delta^{-1}\cdot (u \otimes k)\cdot 
  \delta.
  \end{align*}
 Since $U_p \subset R_u(P_x)$, the element $\delta^{-1}\cdot (u \otimes  k)
  \cdot \delta$ belongs to $R_u(P)$.  
  Thus $\phi(g_pu)$  lies in the $R_u(P)$-orbit of $\phi(g_p)$. 
  
    Let $j \in I(\Q)$. 
    We will show that $\phi(j\cdot g) = 
    \phi(g)$ for every element $g \in I(\A_f)$.    
    By the construction in 
    $\S$\ref{sunif},
   the isogeny $j$ induces a prime-to-$p$ quasi-isogeny 
   \[ \theta : 
    (A_{(j \cdot g )_px}, \lambda_{(j\cdot g)_{p}x}, 
    \eta_{(j \cdot g)^p x} )
    \to 
    (A_{g_px}, 
   \lambda_{g_px}, \eta_{g^px}
   )
    \]
     such that the following   diagram of isomorphisms of  Dieudonn\'{e} modules is commutative: 
    \begin{align}\label{D(j)}
 \begin{CD}
 \D(\mathcal A_x) @>{g_p}>> g_p \D(\mathcal A_x) 
  @=   
  \D(A_{g_px})
  \\
  @\vert @V{j_p}
 VV @V{\D(\theta)}VV 
  \\
  \D(\mathcal A_x) @>{j_pg_p}>> j_p g_p \D(\mathcal  A_x) 
  @=
  \D(A
  _{(j \cdot g)_px}).
  \end{CD}
  \end{align}
    Here, we write  
 $j_p \in I(\Q_p)$ for the image of $j$.    
  Now we define a prime-to-$p$ quasi-isogeny 
   \[\psi : 
   (\mathcal A'_{\varphi(\iota_x(j\cdot g))},
    \lambda_{\varphi(\iota_x(j\cdot g))}, 
    \eta_{{\sf K}',\varphi(\iota_x(j\cdot g))}) 
    \to  
    (\mathcal  A'_{\varphi(\iota_x(g))}, \lambda_{\varphi(\iota_x(g))},\eta_{{\sf K}',\varphi(\iota_x(g))})\]
      by the following  commutative diagram: 
\[
 \begin{CD}
 \mathcal A'_{\varphi(\iota_x(j\cdot g))}  @>>>
   A_{(j\cdot g)_px}
  \\
   @V{\psi}VV @V{\theta}VV
   \\
    \mathcal A'_{\varphi(\iota_x(g))} @>>> 
     A_{g_p x}  
  \end{CD} 
 \]
 Here, the horizontal arrows are the canonical prime-to-$p$ quasi-isogenies as above.  
 But 
 by Proposition \ref{kis}, we have 
    $\iota_x(j\cdot g)=\iota_x(g)$. 
  Since ${\sf K}'^p$ is small,  the endomorphism $\psi$ is the identity morphism.  
  Hence one has a  commutative diagram
 \[
 \begin{CD}
  \D(A_{g_px})
  @>{\sim}>>
  \D(\mathcal A'_{ \varphi(\iota_x(g))})
  @=
  \D(\mathcal A_{\iota_x(g)})
  \\
   @V{\D(\theta)}VV 
 @{\vert}
 @{\vert}
  \\
  \D(A
  _{(j \cdot g)_px}) @>{\sim}>> 
  \D(\mathcal A'_{ \varphi(\iota_x(j \cdot g))})
  @=
  \D(\mathcal A_{\iota_x
  (j \cdot g)
  }).
  \end{CD}
  \]
 This diagram and (\ref{D(j)}) show that the element $\phi(j \cdot g)$ coincides with $\phi(g)$.
 Thus we see the map (\ref{emb}) is well-defined.  
 
  Next we show the map is  injective. 
   Suppose that 
   $\phi(g)\cdot R_u(P)=\phi(g')\cdot R_u(P)$ for $g,g'\in I(\A_f)$. 
   Then we have
  $\iota_x(g)=\iota_x(g')$.
   Hence, by Proposition \ref{kis},  we have 
     $g'=j \cdot g h_ph^p$ for some  elements $j\in I(\Q)$, $h_p \in I(\Z_p)$, and $h^p\in U^p$. 
   It suffices to show that $h_p$ belongs to $U_p$. 
  We have 
  \begin{align*}
  \phi(g') & =((j \cdot g)_p h_p \otimes k)\delta 
   =((j \cdot g)_p\otimes k)\delta \delta^{-1}(h_p\otimes k)\delta
  \\
  &=\phi(j\cdot g)\delta^{-1}(h_p\otimes k)\delta=\phi(g)
  \delta^{-1}(h_p \otimes k)\delta.
  \end{align*}
  Hence the assumption implies 
  that $\delta^{-1}(h_p \otimes k)\delta$ belongs to $R_u(P)$, and further this implies that 
  $h_p \in R_u(P_x)$. 
  Thus the element  $h_p $ belongs to $U_p=\ker (\varpi : I(\Z_p) \to P_x/R_u(P_x))$, and  
  we see the map is injective.   
 \end{proof}
 For an open compact ${\sf K}^p_1 \subset {\sf K}^p$, we set $U_1:=U_p\cdot (I(\A_f^p)\cap {\sf K}_1^p)$.  
Then the projection $\pi_{{\sf K}_1/{\sf K}}$ and the embedding (\ref{emb})  induce a commutative diagram  
 \[
 \begin{CD}
  I(\Q) \backslash I(\A_f) /U_1 @>>> 
 I(\Q) \backslash I(\A_f) /U
 \\
 @VVV 
 @VVV
 \\
 \mathcal L_{\mathcal Z_{{\sf K}}(x)} 
 @>>> 
 \mathcal L_{\mathcal Z_{{\sf K}_1}(x)}.
 \end{CD}
 \] 
 Further, the right actions by an element $h \in I(\A_f^p)$ give a  commutative diagram 
  \begin{align*}
  \begin{CD}
 I(\Q) \backslash I(\A_f)/U   @>{h}>>
   I(\Q) \backslash I(\A_f)/h^{-1}U h
  \\
  @VVV 
  @VVV 
  \\
    \mathcal L_{\mathcal Z_{\sf K}(x)}
 @>{h}>> 
   \mathcal L_{\mathcal Z_{h^{-1}{\sf K}h}(x)}.
   \end{CD}
  \end{align*}
\section{Algebraic modular forms$\pmod p$ and Hecke eigensystems}
 \subsection{Algebraic modular forms}\label{ssalg}
 First we review some basic properties of algebraic modular forms. 
A reference is the original work of Gross \cite{Gross}. 
 Set $U^p=I(\A_f^p)\cap {\sf K}^p$ and $U=U_pU^p$. 
Let $V_{\tau}$ be a finite dimensional $k[I(p)]$-module. 
We write  ${\rm Map}(I(\Q)\backslash I(\A_f) /U, V_{\tau})$ for the $k$-vector space of  functions  from the finite set  $I(\Q)\backslash I(\A_f) /U$ to $V_{\tau}$. 
 For a function $f$ and for an element $g \in I(p)$, we define a function $f^g$ by the rule   
 \[f^{g}(x)=g^{-1}f(x \cdot g^{-1}),  \ x \in I(\Q)\backslash I(\A_f) /U\] 
where  $I(p)=I(\Z_p)/U_p$ acts on $I(\Q)\backslash I(\A_f) /U$ through the inclusion  $I(\Z_p)/U_p \subset I(\Q_p)/U_p$. 
Then $f \mapsto f^g$  induces a right $I(p)$-action on the space ${\rm Map}(I(\Q)\backslash I(\A_f) /U, V_{\tau})$. 
  \begin{defn}  
We define the space of \emph{algebraic modular forms} (mod $p$) of weight $V_{\tau}$ and level $U$ on $I$ as the $I(p)$-invariant space:  
\begin{align*}
\begin{split}
M^{\rm{alg}}(I, V_{\tau}, U):=  
{\rm Map}(I(\Q)\backslash I(\A_f) /U, V_{\tau})^{I(p)}.
  \end{split}
  \end{align*} 
  \end{defn} 
  \begin{prop}[{cf.~{\cite[Prop.~4.5]{Gross}}}]\label{propGross}
{\rm (1)} If we fix representatives 
 $\{g_{\alpha}\}$ 
  for the classes in the double coset $I(\Q) \backslash I(\A_f)/I(\Z_p)U^p $, 
then the map $f \mapsto (\ldots, f(I(\Q)g_{\alpha}U), \ldots)$ gives an isomorphisms of $k$-vector spaces
\[M^{{\rm alg}}(I, V_{\tau}, U) \xrightarrow{\sim}  V_{\tau}^{\oplus {I(\Q) \backslash I(\A_f)/I(\Z_p)U^p }}.\]
In particular one has  
\begin{align}\label{dimalg}
\dim_kM^{{\rm alg}}(I, V_{\tau}, U)=\lvert \mathcal Z(x) \rvert \cdot  \dim_k V_{\tau}.\end{align}

{\rm (2)} If we regard 
 $M^{{\rm alg}}(I, *, U)$ 
 as a functor from the category of finite dimensional  $k[I(p)]$-modules to the category of finite dimensional $k$-vector spaces, then this is an exact functor. 
\end{prop} 
\begin{proof}
We recall that ${\sf K}^p$ is assumed to be sufficiently small. 
Therefore, the group $U^p:= I(\A_f^p)\cap {\sf K}^{p}$ satisfies $I(\Q) \cap 
g_{\alpha}(I(\Z_p)U^p) g_{\alpha}^{-1}=\{1\}$ for each $g_{\alpha}$. 
Hence  
 any element $g \in I(\A_f)$ can be uniquely written as $g=
jg_{\alpha} h_ph^p$ for some representative $g_{\alpha}$,    $j \in I(\Q)$, $h_p \in I(\Z_p)$, and $h^p \in U^p$. 
We put $\bar{h}_p:=h_pU_p \in I(p)$.  
By the uniqueness, the   assignment
  \[{\bf{v}}=(\ldots, v_{\alpha}, \ldots)  \mapsto [f_{{\bf v}} : I(\Q)gU \mapsto\bar{h}_p^{-1}\cdot 
 v_{\alpha}]\] 
  induces a well-defined map $ V_{\tau}^{\oplus{I(\Q) \backslash I(\A_f)/I(\Z_p)U^p }}
   \to M^{\rm alg}(I, V_{\tau}, U)$.  This gives an  inverse of the above map, and  
   thus the assertion (1) follows. 
   
  The left exactness of $M^{{\rm alg}}(I,  * , U)$ 
   follows from the definition,  without the assumption on ${\sf K}^p$.  
  Formula  (\ref{dimalg}) implies the exactness.  
\end{proof}
Next we interpret  algebraic modular forms  as  $L$-invariant vector-valued functions on the scheme $\mathcal L_{\mathcal Z(x)}$. 
We use   
 the embedding $I(\Q)\backslash 
 I(\A_f)/U \hookrightarrow \mathcal L_{\mathcal Z(x)}$ constructed in Proposition \ref{propemb}. 
  We fix a section $\delta$ and the identification  $  
  L_x \cong L$ as in (\ref{P=P_x}). 
Under this identification,    we regard $I(p)$ as a subgroup of $L$. 
 \begin{prop}\label{alg} 
For a rational representation $V$ of $L$ over $k$,   there is an isomorphism of $k$-vector spaces 
 \[ {\rm Map}(\mathcal L_{\mathcal Z(x)}, V)^{L}  
 \xrightarrow{\sim}
 M^{\rm{alg}}\left(I, \Res^{L}_{I(p)}
 V ,U\right). 
 \]
 \end{prop}
 \begin{proof}
 Now we write $\widetilde{\mathcal Z(x)}$ for the image of the embedding  $I(\Q)\backslash I(\A_f)/U  \rightarrow \mathcal L_{{\mathcal Z}(x)}$ in  (\ref{emb}).  
 Then the natural morphism $\widetilde{\mathcal Z(x)} \to \mathcal Z(x)$ is an $I(p)$-torsor, and in particular  one has
  $\widetilde{\mathcal Z(x)} /I(p)=\mathcal Z(x)=
  \mathcal L_{\mathcal Z(x)}/L$. 
 Hence the restrictions of  functions to $\widetilde{\mathcal  Z(x)}$   
 induce an isomorphism
 \begin{align*}
 {\rm Map}\left(\mathcal L_{\mathcal Z(x)}, V\right)^{L} 
 &\xrightarrow{\sim}
 {\rm Map}\left(\widetilde{\mathcal Z(x)}, \Res_{I(p)}^{L}V\right)^{I(p)}
 =
 M^{\rm alg}
 \left(I, \Res_{I(p)}^{L} V, U\right).  
 \end{align*} 
 \end{proof} 
 Combining Lemma \ref{fun=aut} and Proposition \ref{alg}, 
 one has an isomorphism 
 \begin{align}\label{cl}
 H^0(\mathcal Z(x), \mathscr V(V)) \cong M^{\rm alg}(I, \Res_{I(p)}^LV,U).
 \end{align}
 \subsection{Hecke eigensystems on basic central leaves} 
 Now we assume that the point $x$ is lying on the basic Newton stratum $\mathcal N_b$. 
Then we have  $\mathcal Z(x)=\mathcal C(x)$ and $I(\A_f^p)=G(\A_f^p)$ by Theorem  \ref{basic}.
In particular one has $U^p={\sf K}^p$. 

Let $h\in G(\Q_v)$ for an  unramified finite place $v \neq p$ of $G$.  
Each characteristic function of ${\sf K}_vh{\sf K}_v$  defines a linear operator $T_h$ of $M^{\rm{alg}}(I, 
 V,U)$ as follows. 
Write 
${\sf K}_v h {\sf K}_v=\bigsqcup_{i} h_i {\sf K}_v$ as a disjoint union of a finite number of single cosets. 
For $f \in M^{{\rm alg}}(I, V, U)$ we set a function $T_h(f)$ as   
 \[T_h(f)(g):=\sum_{i}f(gh_i),  \ g \in G(\A_f).\]
 Then $T_h(f)$ defines an element in $M^{\rm alg}(I, V, U)$. 
 By extending linearly to $\mathcal H_v$ and to the restricted tensor product for  all $v$, one  get an  $\mathcal H$-module structure  of $M^{\rm alg}(I, V, U)$. 

By the construction,  this $\mathcal H$-module structure coincides with the one induced by the projections and the right actions
\begin{align*}
\pi_{{\sf K}_1/{\sf K}} & : I(\Q) \backslash I(\A_f) /U_1  \to 
 I(\Q) \backslash I(\A_f) /U, 
 \\
h & :  I(\Q) \backslash I(\A_f) /U  \to 
 I(\Q) \backslash I(\A_f) /h^{-1}Uh\end{align*}
 for compact open subgroups ${\sf K}^p_1 \subset {\sf K}^p$ and for  elements  
  $h \in G(\A_f^p)$. 
 Hence, the isomorphism (\ref{cl}) is compatible with $\mathcal H$-actions. 

Now 
we show that  systems of Hecke eigenvalues appearing in  automorphic forms on $S_{\sf K}$ are  the same as those appearing on the  central leaf $\mathcal C(x)$ for any $x \in \mathcal N_b(k)$. 
We write $D$ for the boundary divisor of $\mathscr S_{\sf K}$ in $\mathscr S_{\sf K}^{\Sigma}$, and $\omega=\det \Lie^t(\overline{\mathcal A})$ for  the Hodge line bundle on $\mathscr S^{\Sigma}_{\sf K}$ ($\S$\ref{ext}). 
The following is a slight generalization of the result of Goldring and Koskivirta in \cite[$\S$11]{GK}.  
\begin{thm}\label{GK}
We assume the condition $(\star)$ if $(G, X)$ is neither of compact-type, nor of PEL-type:  

$(\star)$ There exists a $G(\A_f^p)$-equivariant Cartier divisor $D'$ such that $D'_{\rm{red}}=D$ and $\omega^k(-D')$ is ample on $\mathscr S_{\sf K}^{\Sigma}$ for all $k \gg 0$.

 Then each of the $\mathcal H$-modules 
\begin{align}\label{modules}
\begin{split}
& \bigoplus_{\xi \in X_{+}^*}
  H^0( S_{\sf K}, \mathscr V(\xi)), \   
\bigoplus_{\xi\in X_{+}^*} H^0( S_{\sf K}^{\Sigma}, \mathscr V^{\rm{sub}}(\xi)),
\\
& \bigoplus_{\xi \in X_{+}^*} H^0( S_{\sf K}^{\Sigma}, \mathscr V^{\rm{can}}(\xi)), \  
  \bigoplus_{\xi \in X_{+}^*} H^0(\mathcal C(x), \mathscr V(\xi))
  \end{split}
 \end{align}
 for an $x\in \mathcal N_b(k)$ 
 admits precisely the same systems of Hecke eigenvalues. 
\end{thm}
\begin{proof} 
By the work of C.~Zhang in \cite{Zhang}, 
the special fiber $S_{\sf K}$ 
  admits a universal $G$-zip of type $\mu$. 
   It gives rise to a smooth morphism of stacks $\zeta :  S_{\sf K} \to G$-$\rm{Zip}^{\mu}$. 
  By definition, the fibers $ S_{{\sf K},w}:=\zeta^{-1}(w)$ for $w \in G$-Zip$^{\mu}$ are the \emph{Ekedahl-Oort (EO) strata} of $S_{\sf K}$. 
  There is a unique zero-dimensional EO stratum  $S_{{\sf K},e}$,  which is always non-empty \cite{Yu3}. 
 The result in \cite[\S 11]{GK} considered the systems of Hecke eigenvalues  on this stratum $S_{{\sf K},e}$. 
 We have to modify their argument as follows.  
 
  It was shown in \cite[Lem.~11.2.3]{GK} that  
  the module  $\Sym^n((\Lie (G)/\Lie (P))^{\vee}) \otimes V_{\xi}$ has a ``good filtration",  and such a filtration was used in the proof of \cite[Lem.~11.2.3]{GK}. 
On the other hand, our definition of the representations $V_{\xi}$ are different from \cite{GK} as we have remarked in $\S$\ref{autbdl}. 
We can replace a good filtration by a composition series, since a rational representation of $L$ always has a composition series and each  composition factor is isomorphic to $V_{\xi}$ (in our notation) for some $\xi \in X_{+}^*$.    
 
 Furthermore, one sees that their argument for the tower $(S_{{\sf K},e})_{{\sf K}^p}$  can be extended to any tower of non-empty closed subschemes $(\mathcal Y_{{\sf K}})_{{\sf K}^p}$ satisfying the following   assumptions:   
\begin{itemize}
\item[(1)] The tower $(\mathcal Y_{{\sf K}})_{{\sf K}^p}$ is 
 $G(\A_f^p)$-equivariant. 
 More precisely,  the right action by an element $h\in 
  G(\A_f^p)$ induces a bijection 
 $\mathcal Y_{{\sf K}} \xrightarrow{\sim}
  \mathcal Y_{h^{-1}{\sf K}h}$, and the projection $\pi_{{\sf K}_1/{\sf K}}$ for any compact open subgroup  ${\sf K}^p_1 
 \subset 
 {\sf K}^p$  induces an equality 
 $\pi^{-1}_{{\sf K}_1/{\sf K}}(\mathcal Y_{{\sf K}})=\mathcal Y_{{\sf K_1}}$.
\item[(2)] 
Each closed subscheme   $\mathcal Y_{{\sf K}}$ is smooth of dimension zero.   
  \item[(3)] \cite[Lem.~11.2.5]{GK} 
  There is an 
  isomorphism of $\mathcal H$-modules 
  \[H^0(\mathcal Y_{{\sf K}}, \mathscr V(\xi)) \xrightarrow{\sim}
   H^0(\mathcal Y_{{\sf K}}, \mathscr
   V(\xi) \otimes \omega^{N})\] 
   for some integer $N\geq 1$. 
  \end{itemize}
We recall here how these assumptions are used in  
\cite[Lemmas  11.2.4, 2.5,  2.6]{GK}. 
Assumption  (1) implies that the ideal sheaf $\mathcal I=\mathcal I_{{\sf K}}$ of $\mathcal Y_{\sf K}$ in $S_{\sf K}$, and also 
  the  restrictions   of 
$\mathscr V(\xi)$, $\Omega_{S_{\sf K}}^1$, and $\omega$ to $\mathcal Y_{\sf K}$ are $G(\mathbb A_f^p)$-equivariant \cite[$\S$4.1.10]{GK}. 
This further implies that their  cohomology groups are equipped with structures of 
$\mathcal H$-modules. 
Assumption (2) implies that $\mathcal I^j/\mathcal I^{j+1}=\Sym^j(\mathcal I/\mathcal I^2)$ for $j \in \Z_{\geq 0}$, 
and the conormal exact sequence of differentials  simplifies to an isomorphism $\mathcal I/\mathcal I^2 \cong \Omega^1_{S_{\sf K}} \vert_{\mathcal Y_{\sf K}}$ on $\mathcal Y_{\sf K}$. 
Assumption (3) is used crucially in the proof of \cite[Lemma 2.6]{GK}. 
Here we recall that $\mathcal Y_{{\sf K}}$ does not meet the boundary $D$ by definition, and this implies  $\omega^k(-D') \vert_{\mathcal Y_{{\sf K}}}=\omega^k \vert_{\mathcal Y_{{\sf K}}}$. 
We also note that the argument given in \cite[Proposition 24]{Ghitza} and repeated in \cite[Proposition 5.17]{Reduzzi} seems to contain a mistake related to compactifications, which was pointed out in \cite[Remark 11.2.1]{GK} and corrected in \cite[Lemma 2.6]{GK}.

 Hence it suffices to show that the tower of central leaves $(\mathcal C_{{\sf K}}(x))_{{\sf K}^p}$ satisfies these assumptions. 
Since $x$ is basic, one has 
$I(\A_f^p)=G(\A_f^p)$ and  
 $\mathcal C_{\sf K}(x)=\mathcal Z_{\sf K}(x)$  
  by Theorem \ref{basic} and Corollary \ref{c=z}. 
  Therefore, assumptions (1) and (2)  follow from Proposition \ref{Z_x}. 
  
Now we give two proofs of assumption  (3) for $\mathcal C(x)=\mathcal C_{{\sf K}}(x)$. 
The first one uses the \emph{Hasse invariants} on  EO strata. 
Recall that the structure of  central leaves is finer than the structure of the EO strata. 
In particular,  $\mathcal C(x)$ is contained in the EO stratum $S_w=S_{{\sf K},w}$ for some $w \in G$-${\rm Zip}^{\mu}$.  
  Let $\overline{S}_{w}$ be the Zariski closure of $S_{w}$ in $S_{\sf K}$. 
  By \cite[Cor. 4.3.5]{GK}, there exist a positive integer $N$ and a section $h_w \in H^0(\overline{S}_w, \omega^N \lvert_{\overline{S}_w})$ which is $G(\A_f^p)$-equivariant and whose non-vanishing locus is precisely $S_w$.  
  Then its restriction $h_w \vert_{\mathcal C(x)}$ to $\mathcal C(x)$ is $G(\A_f^p)$-equivariant, and does not vanish anywhere on $\mathcal C(x)$. 
  Therefore, the multiplication by $h_w \vert_{\mathcal C(x)}$  induces the desired isomorphism. 
  
  The second proof is similar to Ghitza's \cite{Ghitza}. 
 By \cite[$\S$4.1.11]{GK},  there is a  structure $k_{\omega}$ on $k$ of a  one dimensional  rational representation of $L$  such that $\omega \cong \mathscr V(k_{\omega})$ as $G(\A_f^p)$-equivariant line bundles. 
  Recall that we regard $I(p)$ as a subgroup of $L$. 
  Since $I(p)$ is finite,  there exists an embedding of a  finite subfield $\F_q \hookrightarrow k_{\omega}$ such that the action of  $I(p)$ on $k_{\omega}$ factors through $\F_q^{\times}$.  
  Therefore, by (\ref{cl}) one has identifications 
  \begin{align*}
 H^0(\mathcal C(x), \mathscr V(\xi)) & \cong M^{\rm alg}(I, \Res^L_{I(p)} V_{\xi}, U)
  \\
  & =  M^{\rm alg}(I, \Res^L_{I(p)}  
  (V_{\xi}
  \otimes 
  k_{\omega}^{\otimes q-1}), 
  U)
   \cong 
  H^0(\mathcal C(x), 
  \mathscr V(\xi)
   \otimes \omega^{q-1}
   ).
  \end{align*}
\end{proof}  
\subsection{Main theorem}
We write  ${\rm Irr}(I(p))$ for the set of isomorphism classes of simple $k[I(p)]$-modules.   
 We use the following lemma. 
 \begin{lemma}\label{lift}
For each $V_{\tau} \in {\rm Irr}(I(p))$, 
there exists a rational representation $V$ of $L$ over $k$ such that the restriction $\Res_{I(p)}^{L} V$  contains  $V_{\tau}$ as a  $k[I(p)]$-submodule. 
\end{lemma}
\begin{proof}
Recall that one has an embedding $I(p)\subset L_x  \cong L \subset \GL(\Lambda^*_k)$. 
 Since $I(p)$ is finite,  there is a finite subfield $\F_q$ of $k$ such that 
 $I(p) \subset  \GL(\Lambda^*_{\F_q})$. 
 
 Let  $V_{\pi}$ be the quotient of the induced 
 module $\Ind_{I(p)}^{\GL(\Lambda^*_{\F_q})}V_{\tau}$ by a  maximal proper $k[\GL(\Lambda_{\F_q}^*)]$- submodule. 
  Then $V_{\pi}$ is a simple $k[\GL(\Lambda^*_{\F_q})]$-module. 
  By  Frobenius reciprocity, 
  the restriction 
  $\Res_{I(p)}^{\GL(\Lambda^*_{\F_q})} V_{\pi}$ contains $V_{\tau}$ as a $k[I(p)]$-submodule. 
 Further, by \cite[Prop.~26]{Ghitza}, there is a structure of (irreducible) rational  representation  of $\GL(\Lambda^*_k)$ on $V_{{\pi}}$ which lifts the $k[\GL(\Lambda^*_{\F_q})]$-module structure of $V_{\pi}$. 
 If we write $V_{\overline{\pi}}$ for this representation, 
 then the restriction 
 $V:=\Res_{L}^{\GL(\Lambda_k^*)} 
 V_{\overline{\pi}}$ satisfies the desired property since 
\[\Res_{I(p)}^{L} V=
\Res_{I(p)}^{\GL(\Lambda^*_k)} V_{\overline{\pi}}=\Res_{I(p)}^{\GL(\Lambda^*_{\F_q})} V_{\pi}\supset V_{\tau} . \] 
\end{proof} 
\begin{thm}\label{Hevalg}
The systems of Hecke eigenvalues which appear in the space
\[\bigoplus_{\xi\in X_{+}^*}H^0(\mathcal C(x), \mathscr V(\xi))\]
 for an $x\in \mathcal N_b(k)$ 
  are the same as those which appear in  the space 
  \[
  \bigoplus_{V_{\tau} \in {\rm Irr}(I(p))}
  M^{{\rm alg}}(I, V_{\tau}, U)\]
  of algebraic modular forms$\pmod p$ 
 for $I$ having level $U=U_p{{\sf K}}^{p}$. 
\end{thm}
\begin{proof}
First let $\xi \in X_{+}^{*}$ and suppose that  $f \in H^0(\mathcal C(x), \mathscr V(\xi))$ is a Hecke eigenform and $(b_T)$ is the system of 
 Hecke eigenvalues of $f$. 
 Under the isomorphism (\ref{cl}), we 
can regard $f$ as a Hecke eigenform in the space $M^{{\rm alg}}(I, \Res_{I(p)}^{L} V_{\xi}, U)$. 
Recall that the functor $M^{\rm alg}(I, *, U)$ on finite $k[I(p)]$-modules $*$  is exact by Proposition \ref{propGross}. 
Therefore, for any composition series 
 of   
$\Res_{I(p)}^{L} V_{\xi}$, 
  one can find a composition factor $V_{\tau}$   
  such that $(b_T)$ appears as a system of Hecke eigenvalues in  $M^{{\rm alg}}(I, V_{\tau},U)$. 
 
 On the other hand, 
 let $V_{\tau}$ be a  simple $k[I(p)]$-module. 
 Then by Lemma \ref{lift}, there is a  rational representation $V$ of  $L$  such that the restriction $\Res_{I(p)}^{L} V$ has $V_{\tau}$ as an  $k[I(p)]$-submodule. 
Therefore one has a $\mathcal H$-equivariant inclusion 
\[ M^{{\rm alg}}
(I, V_{\tau}, U) \subset M^{{\rm alg}}
(I, \Res_{I(p)}^{L} V, U).\]
Note that the functor $\Res^{L}_{I(p)} (*)$ on rational representations of $L$ is exact. 
Therefore, for any system $(b_T)$ of Hecke eigenvalues arising  from $M^{{\rm alg}}
(I, V_{\tau}, U)$, there exists a composition factor $V'$ of $V$ such that $(b_T)$ appears also in $M^{{\rm alg}}
(I, \Res_{I(p)}^{L}V', U)$. 
By (\ref{cl}),  
 the system $(b_T)$ also appears in the space 
 $H^0(\mathcal C(x), 
 \mathscr V(\xi))$ 
 where $\xi \in X_{+}^*$ is   the dominant weight satisfying
  $V_{\xi} \cong V'$. 
 \end{proof}
 Theorems \ref{GK} and \ref{Hevalg} imply Theorem \ref{intro}. 
\begin{cor}\label{number}
Under the assumption of Theorem \ref{intro}, 
the  number of distinct systems of Hecke eigenvalues  appearing in the spaces (\ref{mf}) 
 is no larger than  
\[\vert \mathcal C(x) \vert \cdot \sum_ {V_{\tau} \in {\rm Irr} (I(p))} \dim_k V_{\tau}.\]
\end{cor}
\begin{proof} By
Theorem \ref{intro}, that  number is no larger than 
the dimension of the direct sum $\bigoplus_{V_{\tau} \in {\rm Irr}(I(p))}
  M^{{\rm alg}}(I, V_{\tau}, U)$ over $k$. 
Hence the assertion follows from  formula  (\ref{dimalg}).  
\end{proof}

 \section{Upper bound of the number of Hecke eigensystems of   totally indefinite quaternionic Shimura varieties} 
 Let $p$ is a prime number, possibly $p=2$. 
 In this section we apply Corollary \ref{number} to totally indefinite quaternionic  Shimura varieties  (type C), and give an explicit upper bound of the number of Hecke eigensystems appearing in the spaces of automorphic forms (mod $p$). 
\subsection{Hecke eigensystems of totally indefinite quaternionic Shimura varieties}
Let $\scrD$ be a tuple $(B,*,O_B, V,\psi, \Lambda, h_0)$, where
\begin{itemize}
\item 
  $B$ is a totally indefinite quaternion algebra over a totally real field $F$ of degree $d$ over $\mathbb Q$;
  \item $*$ is  a positive involution on $B$;
  \item  
 $O_B$ is a maximal order stable under the involution $*$; 
 \item $V$ is a finite faithful left $B$-module;
\item $\psi:V\times V\to \Q$ is a non-degenerate alternating form such that
  $\psi(bx,y)=\psi(x,b^*y)$ for all $x,y\in V$ and $b\in B$; 
  \item $\Lambda$ is a full $O_B$-lattice in $V$; 
\item $h_0:\C \to \End_{B\otimes \R}(V_\R)$ is an $\R$-linear algebra
  homomorphism, % where $V_\R:=V\otimes_\Q \R$,  
  such that 
\[ \phi(h_0(i)x, h_0(i)y)=\psi(x,y), \quad \forall\, x,y \in
 V_\R=V\otimes\R, \]
%for any $x, y\in V_\R=V\otimes\R$, one has 
%  $\phi(h_0(i)x, h_0(i)y)=\psi(x,y)$ and 
and the symmetric form  $(x,y)=\psi(x,h_0(i)y)$ is
  definite (positive or negative) on $V_\R$.   
 \end{itemize}
 In addition, for any prime number  $\ell$ we assume that 
  \begin{itemize}
     \item[(a)]  $O_{B}\otimes_{\Z} \Z_{\ell}$
     is a  
       maximal order in $B\otimes_\Q \Q_{\ell}$;
     \item[(b)] $\Lambda\otimes_\Z \Z_{\ell}$ is self-dual with
       respect to the pairing $\psi$.   
    \end{itemize}
    Let $G=\GU_B(V,\psi)$ be the $\Q$-group of $B$-linear
$\psi$-similitudes on $V$. 
The
group of its $R$-valued points for a  commutative $\Q$-algebra $R$ 
 is defined by 
\begin{equation}
  \label{eq:P.8}
  G(R)=\{g\in \End_B(V_R)\, |\, \exists\, c(g)\in R^\times
  \ \text{s.t.}\ \psi(gx,gy)=c(g)\psi(x,y), \ \forall\, x,y\in
  V_R \}
\end{equation}
where $V_R=V\otimes_\Q R$. 
 Further, $O_B$ and $\Lambda$ give a model $G_{\mathbb Z}=\GU_{O_B}(\Lambda, \psi)$ over $\mathbb Z$. 
 We fix an integer $N\geq 3$ with $p \nmid N$,  and  
 we define a compact open subgroup  ${\sf K}^p(N) \subset G(\mathbb A_f^p)$ by
 \[{\sf K}^p(N)=\ker \Big(G_{\mathbb Z}(\widehat{\Z}^p) \to 
 G_{\mathbb Z}(\widehat{\Z}^p/N\widehat{\Z}^p)\Big).\]
 We set ${\sf K}_p=G_{\Z}(\Z_p)$ and ${\sf K}={\sf K}_p \cdot {\sf K}^p(N)\subset G(\mathbb A_f)$.  
 
 Let $m=\dim_B V$, and let $\mathbf M_{\sf K}=\mathbf M_{\sf K}(\scrD)$ be the moduli problem for classifying $O_B$-abelian  schemes of dimension $2dm$  up to prime-to-$p$ isogeny with  $\Z_{(p)}^{\times}$-polarizations, level ${\sf K}^p(N)$ structure, and  the determinant condition. 
By \cite{Kottwitz,RZ}, 
 this functor $\mathbf M_{\sf K}$ is represented by a quasi-projective scheme (denoted again by) $\mathbf M_{\sf K}$ over $\Z_{(p)}$. 
 Further,  assumptions (a) and (b) for all primes $\ell$ imply that  $\mathbf M_{\sf K}$ is identified with   the moduli space classifying abelian schemes of dimension $2dm$ with {\it principal} polarizations, $O_B$-multiplications, and level ${\sf K}^p(N)$  structure, where two objects are said to be  isomorphic if there exists an isomorphism of abelian schemes compatible with the additional structures (cf. \cite[Prop.~1.4.3.4]{Lan}).        
 
 We write $\mathbf M_{{\sf K}}^{\rm sp}$ for the subset of $\mathbf M_{{\sf K}}(k)$ consisting of  points whose underlying abelian varieties are superspecial. 
Then $\mathbf M_{{\sf K}}^{{\rm sp}}$ is non-empty.  
We can reduce this problem to   constructing a Dieudonn\'{e} module with certain additional structures (cf. \cite{xue-yu:ppas, 2003Fourier, slope, yu:D}).  We omit details.\footnote{The reader is referred to an earlier version \cite{TY} for the proof.}

From now on, we assume $p$ is unramified in $B$.
Then the moduli space $\mathbf M_{{\sf K}}$ has good reduction at $p$ and can be regarded as the canonical integral model of a Shimura variety of Hodge type with a hyperspecial level ${\sf K}_p$ at $p$. 
We choose a base point $x\in \mathbf M_{\sf K}^{{\rm sp}}$, 
and consider  
   its associated inner form $I$ of $G$  as in $\S$\ref{sunif}.  
 
Let $\xi \in X_{+}^*$ and $\mathscr V(\xi)$ be the automorphic bundle of weight $\xi$. 
Let $\Sigma$ be a finite, admissible rpcd ($\S$\ref{ext}). 
Let $\mathbf M_{\sf K}^{\Sigma}$ be the toroidal compactification of $\mathbf M_{\sf K}$ with respect to $\Sigma$, constructed by Lan \cite{Lan}. 
We remark that Lan does not excluded the case $p=2$ when $\mathbf M_{{\sf K}}$ is of  type C. 
Also, the proofs in previous sections  work for $p=2$ with the following modifications: 
For the  Rapoport-Zink uniformization, 
 we can apply the original work \cite[Chap.~6]{RZ} to the variety $\mathbf M_{{\sf K}}$. 
 Especially, Theorem \ref{basic} can be  replaced by \cite[Thm.~6.30]{RZ}. 
 Note that for general Hodge type case, Kisin \cite{Kisin2} and Howard-Pappas \cite{HP} excluded the case $p=2$. 
The proof of Theorem \ref{GK} works also  for $p=2$ if we choose the second proof of assumption  (3).  
(However,  
we do not know whether the construction of Hasse  invariants given by  Goldring and Koskivirta in \cite[$\S$4]{GK} can be extended to the case $p=2$. 
See also \cite[Rem.~11.2.7]{GK}.) 
Hence Theorem \ref{intro} holds for $\mathbf M_{{\sf K}}$ for any prime $p$, i.e. the systems of Hecke eigenvalues which appear  in  each of  the $\mathcal H$-modules 
\[\bigoplus_{\xi \in X_{+,L}^*(T) 
}H^0(\mathbf M_{{\sf K}} \otimes k, 
\mathscr V(\xi)), \bigoplus_{\xi \in X_{+,L}^*(T) 
}H^0(\mathbf M_{{\sf K}}^{\Sigma} \otimes k, 
\mathscr V^{\bullet}(\xi))\]
 for $\bullet \in \{{\rm can, sub}\}$  are the same  as those  which appear in  the space of algebraic modular forms on $I$  of level $U=U_p\cdot{\sf K}^p(N)$ and of varying  weight $V_{\tau} \in {\rm Irr}(I(p))$. 
 \begin{remark}
When $F=\Q$, $B=\Mat_2(\Q)$, and $m=\dim_B V=1$, then the Morita  equivalence implies that $\mathbf M_{{\sf K}}$ parameterizes abelian varieties of dimension $dm=1$, i.e.  $\mathbf M_{{\sf K}}$ is the modular curve.  
In particular, the boundary has codimension one. 
Otherwise $\mathbf M_{{\sf K}}$ is compact or has the  boundary with codimension  larger than one, and hence  
the Koecher principle holds for $\mathbf M_{{\sf K}}$ by \cite[Thm.~2.3]{Lan3}:
\[H^0(\mathbf M_{\sf K}^{\Sigma} \otimes k, \mathscr{V}(\xi)^{\rm{can}}) \cong 
H^0(\mathbf M_{\sf K} \otimes k, \mathscr{V}(\xi)).\] 
\end{remark}
   By \cite[Lem.~4.2 (1)]{Yu2}, the central leaf $\mathcal C(x)$  for any point $x \in \mathbf M_{{\sf K}}^{{\rm sp}}$ coincides with 
 $\mathbf M_{{\sf K}}^{{\rm sp}}$.    
 Note that its cardinality 
 $\lvert \mathcal C(x) \rvert=
 \lvert \mathbf M_{{\sf K}}^{{\rm sp}}\rvert$ is the smallest among all central leaves. 
 
   We write $\mathscr N(B,N,p)$ for the number of the above distinct systems of Hecke eigenvalues.
Then  Corollary \ref{number} implies an inequality 
 \begin{align}\label{upp}
 \mathscr N(B,N,p) \leq \vert \mathbf M_{{\sf K}}^{{\rm sp}} \vert \cdot \sum_ {V_{\tau} \in {\rm Irr} (I(p))} \dim_k V_{\tau}.
 \end{align} 
 %For any $\underline{A}$ and any prime $\ell \neq p$, we write $T_{\ell}(\underline{A})$ for the associated Tate module with  additional structures $(T_{\ell}(A), \lambda_{\ell}, \iota_{\ell})$, 
%where $\lambda_{\ell}$ is the induced quasi-polarization from $T_{\ell}(A) \to T^t_{\ell}(A)$? and $\iota_{\ell}:O_B \otimes \mathbb Z_{\ell} \to \End(T_{\ell}(A))$ the induced ring monomorphism. 

  %the set of isomorphisms 
 % \ (\rm{resp.} \ \GIsom_k(\underline{A}_1, \underline{A}_2))$ the set of $O_B$-linear  quasi-isogenies (resp. isomorphisms) $\phi : \underline{A}_1 \to \underline{A}_2$ over $k$ such that $\phi^*\overline{\lambda}_2=\overline{\lambda}_1$ as sets of isogenies from $A_1$ to $A_1^t$. 
 
% For any prime $\ell \neq p$, denote by 
% $\GIsog_k(T_{\ell}(\underline{A}_1), T_{\ell}(\underline{A}_2))$
% (resp. $\GIsom(T_{\ell}(\underline{A}_1), T_{\ell}(\underline{A}_2)))$ the set of 
% $O_B\otimes \mathbb Z_{\ell}$-linear  quasi-isogenies (resp. isomorphisms) 
% $\phi : T_{\ell}\underline{A}_1 \to T_{\ell}\underline{A}_2$ over $k$ such that $\phi^*\lambda_2=q\lambda_1$ for some $q$?

%\begin{lemma}\label{16}
  %Any two isogenies $\phi, \tilde{\phi}: (A_0, \lambda_0, i_0) \to (A, \lambda, i)$ are related by $\tilde{\phi}=\phi\circ u$, where $u \in G'(\mathbb Q)$. 
  %\end{lemma}
\subsection{Cardinality of the superspecial locus}\label{card}
Let $D_{p, \infty}$ be the quaternion algebra over $\mathbb Q$ ramified exactly at $\{p, \infty\}$. 
Let $D'$ be the quaternion algebra over $F$ such that 
\[\inv_v(D')=\inv_v((D_{p, \infty} \otimes _{\Q}F)\otimes _F B)\]
 for any place $v$ of $F$. 
Let $\Delta'$ be the discriminant of $D'$ over $F$. 
\begin{thm}
If we write $\zeta_F(s)$ for  the Dedekind zeta function, then we have 
\begin{align}\label{Mass}
\begin{split}
\lvert \mathbf M_{{\sf K}}^{{\rm sp}} \rvert= & \lvert G(\mathbb Z/N\mathbb Z)\rvert  \cdot\frac{(-1)^{dm(m+1)/2}}{2^{md}}
\\ 
 & \cdot \prod_{j=1}^{m}
\left[ \zeta_F(1-2j)\prod_{v \mid \Delta'}(q_v^j+(-1)^j) \cdot \prod_{v \mid p, v \nmid \Delta'}(q_v^j+1)\right].
\end{split}
\end{align}
\end{thm}
\begin{proof}
The moduli space $\mathbf M_{{\sf K}} \otimes k$ is decomposed as the union $\bigsqcup_{n \in (\Z/N\Z)^{\times}} \mathcal M_{n}$ 
where $\mathcal M_{n}$ is the moduli space over $k$ of $2dm$-dimensional principally polarized $O_B$-abelian varieties with a symplectic $O_B$-linear level $N$ structure with respect to the primitive $N$th root of unity  $\exp({\frac{2\pi i n}{N}})$.

The cardinality of the superspecial locus of $\mathcal M_n$ is given by a \emph{mass formula} \cite[Thm.~1.3]{Yu2}, 
where the factor $\lvert G(\Z/N\Z)\rvert$ in (\ref{Mass}) was replaced by the cardinality $\lvert G^1(\Z/N\Z) \rvert$ of the group of $\Z/N\Z$-values of $G^1=\ker(c :  G \rightarrow \mathbf G_m)$. 
Thus the assertion  (\ref{Mass})  follows from  the above decomposition and the equality 
$\lvert G(\Z/N\Z) \rvert = \lvert (\Z/N\Z)^{\times} \rvert  \cdot \lvert G^1_{\Z}(\Z/N\Z)\rvert.$  

Note that in \cite{Yu2} the superspecial locus  of $\mathcal M_{n}$ was implicitly assumed to be non-empty, 
which is true as mentioned above.  
\end{proof}
\subsection{Number of isomorphism classes of  simple $k[I(p)]$-modules}   
 An explicit description of the inner form $I$ associated to a point $x \in \mathbf M_{{\sf K}}^{{\rm sp}}$ was given in   \cite[Proposition 5.2]{Yu2} using Dieudonn\'{e} theory. 
Further, a straightforward computation shows the following description of  the  finite  group  $I(p)$.    
We put $f_v:=[F_v: \Q_p]$ and $q_v:=N(v)=p^{f_v}$ for each place  $v$ of $F$ with $v \mid p$. 
   Note that the condition $v \nmid \Delta'$ (resp. $v \mid \Delta'$) is equivalent to that  $f_v$ is even (resp. $f_v$ is odd). 
    We write $\Res_{\F_{q}/\F_p}\GL_{m, {\F_{q}}}$ for the Weil restriction for a field extension $\F_q/\F_p$. 
 For a $v \mid \Delta'$ and an $\F_p$-algebra $R$, we write $( \ \cdot \ ) ^{\bar{\tau}}$ for the automorphism on $\Mat_m(\F_{q_v^2} \otimes _{\Z_p} R)$ induced by ${\rm Frob}_{q_v} \otimes_{\F_p} \id$. 
   We define a group scheme $L_{I}$ 
   over  $\F_p$  as the subgroup scheme of  
   \begin{align*} \left(\prod_{v \mid p, v \nmid \Delta'}
\Res_{\F_{q_v}/\F_p}\GL_{m, \F_{q_v}} 
\times 
\prod_{v \mid p,  v \mid \Delta'}
\Res_{\F_{q^2_v}/\F_p}\GL_{m, \F_{q^2_{v}}} \right)
\times \mathbb G_{m, \F_p}
\end{align*} 
 whose $R$-points consist of 
 the elements $((A_v)_v,r)$ satisfying 
 $(A_v^t)
  A_v^{\bar{\tau}}=rI_m$ for each $v \mid p, v \mid \Delta'$. 
  \begin{prop}\label{barG}
   One has  isomorphisms of finite  groups 
 \begin{align*}
 I(p) & \xrightarrow{\sim} 
 L_{I}(\F_p) 
 \\
 & \xrightarrow{\sim}
  \{((A_v)_v, r) \in \Big(\prod_{v \mid p, v\nmid \Delta'}\GL_m(\F_{q_v})  \times \prod_{v \mid p, v\mid \Delta'}
 \GL_m(\F_{q^2_v})\Big) \times \F_p^{\times} 
 \\
 & \hspace{6cm} : A_v^t  A_v^{\bar{\tau}}=rI_m \ {\rm for}  \ v\mid \Delta'\}. 
 \end{align*}
  \end{prop}
Now we  recall that an element $g$ in a finite group is said to be $p$-regular if its order is prime-to-$p$. 
By \cite[Cor.~3 of Thm.~42]{SerreL}, the number $\lvert \Irr(I(p))\rvert $ of pairwise non-isomorphic  simple  $k[I(p)]$-modules coincides with the number of $p$-regular conjugacy classes of $I(p)$.  
We will compute this number. 

The derived subgroup $L'_{I}$ of $L_I$ is isomorphic to 
the product 
\[\prod_{v\mid p,v \nmid\Delta'} 
\Res_{\F_{q_v}/\F_p}\SL_{m,\F_{q_v}} \times
 \prod_{v\mid p,v \mid\Delta'}
 \Res_{\F_{q_v}/\F_p}
 {\rm SU}_{m, \F_{q_v}}.\]
 In particular, $L'_{I}$ is simply-connected. 
We write $Z$ for the (connected) center of $L_I$. 
Then, by \cite[Thm. 3.7.6.(ii)]{Carter}, the number of $p$-regular conjugacy classes of $L_{I}(\F_p)$ is equal to $p^l \cdot\lvert Z(\F_p) \rvert$ where $l$ is the semisimple rank of $L_I$.  

Since the ranks of $\SL_{m, \F_{q_v}}$ and $\SU_{m,\F_{q_v}}$   are  equal to $
m-1$, 
the rank of $L'_{I, \F_p}$ is equal to 
 \[\sum_{v\mid p,v \nmid\Delta'}f_v(m-1)+ \sum_{v\mid p,v \mid\Delta'}f_v(m-1)=d(m-1).\]
Note that the group $Z(\F_p)$ of $\F_p$-values is isomorphic to the center  of the finite group $L_I(\F_p)$ (cf.\cite[Prop. 3.6.8]{Carter}). 
Hence, by Proposition \ref{barG}, one has
\begin{align*}
Z(\F_p) \xrightarrow{\sim} \prod_{v\mid p,v \nmid\Delta'}\F_{q_v}^{\times} \times  
\left\lbrace((x_v)_v,r) \in 
 \left(\prod_{v\mid p,v \mid\Delta'} 
 \F_{q_v^2}^{\times} \right)\times \F_p^{\times}: 
 {\rm N}(x_v)=r \right\rbrace
  \end{align*}
  where ${\rm N}$ is the norm map ${\rm N}_{\F_{q^2_v}/\F_{q_v}}
   : \F_{q^2_v}^{\times} \to \F_{q_v}^{\times}, \  x \mapsto x^{\bar{\tau}}x$.
  One has  
\[\lvert {\rm N}^{-1}(r) \rvert  =
(q^2_v-1)/(q_v-1)=q_v+1\]
 for each $r \in \F_p^{\times}$, 
and   
hence we have 
\begin{align*}
\lvert Z(\F_p)\rvert & 
=\prod_{v \mid p, v \nmid \Delta'} \lvert \F_{q_v}^{\times} \rvert
\cdot 
\prod_{r \in \F_p^{\times}} 
\left(\prod_{v \mid p, v\mid \Delta'}
\lvert
{\rm N}^{-1}(r)\rvert\right)
\\
& =(p-1)\cdot \prod_{v\mid p,v \nmid\Delta'}(q_v-1) \cdot \prod_{v\mid p,v \mid\Delta'} (q_v+1).\end{align*}
Therefore, one gets an equality 
\begin{align}\label{irr}
\lvert \Irr (I(p))\rvert =
p^{d(m-1)}\cdot (p-1)\cdot \prod_{v\mid p,v\nmid\Delta'}
(q_v-1) \cdot 
\prod_{v\mid p,v \mid\Delta'}(q_v+1).
\end{align}
\subsection{Upper bound of the dimensions of simple $k[I(p)]$-modules}
Since $L_I$ is connected reductive group over $\F_p$, 
 the finite group $L_I(\F_p)=I(p)$  has a structure of a split $BN$-pair of characteristic $p$ by \cite[$\S$1.18]{Carter}.
Hence   
 the dimensions of  simple $k[I(p)]$-modules are no larger than the order of a $p$-Sylow subgroup of $I(p)$ by 
   \cite[Cor.~3.5 and Cor.~5.11]{Curtis}.\footnote{Curtis assumed that the split $BN$-pair is  ``restricted" (\cite[Def.3.9]{Curtis}), which is a technical condition shown to be implied just by the axioms of split $BN$-pairs by Richen \cite{Richen}.}  

Now we compute the order of a $p$-Sylow subgroup of $I(p)$. 
We set 
\[{\rm U}_m(\F_{q_v}):=\{A \in \GL_m(\F_{q^2_v}) : A^tA^{\bar{\tau}}= I_m\}.\]
 By Proposition \ref{barG}, one has an 
exact sequence 
\[1 \to  \prod_{v\mid p,v \nmid\Delta'}\GL_m(\F_{q_v}) \times \prod_{v\mid p,v \mid\Delta'}
 {\rm U}_m(\F_{q_v})  
 \to I(p) \xrightarrow{{\rm pr}} 
 \F_p^{\times} \to 1,\]
 where ${\rm pr}$ is given by $((A_v)_{v\mid p},r) \mapsto r$. 
   Further, for some $p \nmid n,n' \in \Z_{>0}$ one has  
   \[
\lvert \GL_{m}(\mathbb F_{q_v})\rvert=n\cdot  q_v^{\frac{m(m-1)}{2}} \ {\rm and} \ 
\lvert {\rm U}_m
(\F_{q_v})
\rvert  
=n'\cdot q_{v}^{\frac{m(m-1)}{2}}.\]   
Therefore a $p$-Sylow subgroup of 
$I(p)$
 has order 
\[
\prod_{v \mid p, v\nmid \Delta'}q_v^{\frac{m(m-1)}{2}}\cdot\prod_{v \mid p, v\mid\Delta'}q_v^{\frac{m(m-1)}{2}}
=\prod_{v \mid p}q_v^{\frac{m(m-1)}{2}}
=p^{{\frac{dm(m-1)}{2}}}.
\]

Thus,   
the dimension of a simple $k[I(p)]$-module $V_{\tau}$ satisfies  the   inequality 
\begin{align}\label{dim}
\dim_{k}V_{\tau} \leq 
p^{{\frac{dm(m-1)}{2}}}.
\end{align}
\subsection{Explicit upper bound of the number of Hecke eigensystems}
%\begin{thm}\label{bound}
%Set 
%\[C:=\frac{(-1)^{dm(m+1)/2}}{2^{md}}\cdot \prod_{i=1}^m\left(\zeta_F(1-2i)\prod_{v\nmid p, v\mid \Delta'}(q_v^i+(-1)^i)\right).\]
%\end{thm}
\begin{thm}\label{intro2}
Let $\Delta'$ be the discriminant of  $D'$ over $F$ {\rm ($\S$\ref{card})}. 
We set 
\[C_B:=\frac{(-1)^{dm(m+1)/2}}{2^{md}}\cdot \prod_{i=1}^m\left(\zeta_F(1-2i)\prod_{v\nmid p, v\mid \Delta'}(q_v^i+(-1)^i)\right).\]
 Then one has the inequality
\begin{align*}
& \mathscr N(B,N,p) \leq C_B\cdot \lvert G(\mathbb Z/N\mathbb Z)\rvert\cdot p^{\frac{d(m+2)(m-1)}{2}} \cdot (p-1)
\\
 & \cdot \prod _{v\mid p, v\nmid \Delta'}(q_v-1) \cdot \prod_{v\mid p,v \mid \Delta'}(q_v+1)\cdot\prod_{j=1}^{m}
\left[\prod_{v \mid p, v \nmid \Delta'}(q_v^j+1)\cdot\prod_{v \mid p, v \mid \Delta'}(q_v^j+(-1)^j)  \right].\end{align*} 
In particular, if we keep $F$, $B$, $m$, and $N$ fixed and let $p>1$ vary, then 
\begin{align*}
\mathscr N(B,N,p) = O(p^{dm^2+dm+1}) \  {\rm{as}} \ p \to \infty.\end{align*}
\end{thm} 
\begin{proof}
The inequality follows from  (\ref{upp}), (\ref{Mass}), (\ref{irr}), and (\ref{dim}). 
As $p$ becomes large, one has 
\begin{align*}
 \mathscr N(B,N,p)= & O\left( p^{\frac{d(m+2)(m-1)}{2}}\cdot p 
\cdot 
 \prod_{v\mid p}
 q_v 
\cdot
\prod_{j=1}^{m}
\left(\prod_{v \mid p}q_v^j
\right)\right)
\\
= &
O\left(p^{\frac{d(m+2)(m-1)}{2}}
\cdot p
\cdot p^{d}
\cdot p^{\frac{dm(m+1)}{2}}
\right)
=O\left(p^{dm^2+dm+1}\right). 
\end{align*} 
\end{proof}
\begin{remark}\label{Siegel}
When $F=\Q$,  
$B=\Mat_2(\Q)$ and $*$ is the transpose, then the Morita equivalence implies that the moduli space $\mathbf M_{{\sf K}}$ parameterizes principally polarized abelian varieties of dimension $m$. 
 In other words, $\mathbf M_{\sf K}$ is the Siegel modular variety of genus $m$, and the spaces $H^0(\mathbf M_{\sf K}^{\Sigma} \otimes k, \mathscr V^{{\rm can}}(\xi))$ are regarded as spaces of Siegel modular forms$\pmod p$.   
 In this case, the place $v \mid p$ is just $p$ and the  discriminant is $\Delta' =(p)$. 
% In particular, there is no place $v \nmid p$ with 
 %$v \mid \Delta'$. 
 Hence one has the inequality
\begin{align*}
 \mathscr N(B,N,p) \leq C_B\cdot \lvert \GSp_{2m}(\mathbb Z/N\mathbb Z)\rvert\cdot p^{\frac{(m+2)(m-1)}{2}} \cdot  (p-1)(p+1)
 \cdot\prod_{j=1}^{m}
(p^j+(-1)^j)\end{align*} 
with $C_B
 =\frac{(-1)^{m(m+1)/2}}{2^{m}}
 \cdot 
 \prod_{i=1}^m
 \zeta_{\Q}(1-2i)$. 
This agrees with Ghitza's result \cite[Thm.~1]{Ghitza2}.
\end{remark}
\begin{remark} 
In \cite{Reduzzi}, Reduzzi gave an upper bound of the number of Hecke eigensystems of PEL modular forms$\pmod p$ associated to an imaginary quadratic extension $\Q(\sqrt \alpha)/\Q$ (type A). 
Note that a mass formula for 
 type A has not been known. 
 So   
  he embedded the  superspecial locus into the Siegel modular variety,  and used  Ekedahl's mass formula \cite{Ekedahl}  to give a bound. 
Also,  a mass formula for definite quaternionic type  (type D) is not known. 
 If one has  mass formulae for type A or D, it should be possible to obtain sharper bounds of the numbers of Hecke eigensystems of these types.
\end{remark}

Conflict of interest statement: on behalf of all authors, the
corresponding author states that there is no conflict of interest. 

Data availability statement:all data generated or analysed during this
study are included in this published article.


\begin{thebibliography}{9}
\newcommand{\LNM}[1]{Lecture Notes in Math., vol. #1, Springer-Verlag}
\def\mathann{{\it Math. Ann.}}
%\bibitem{BBM}
%P. Berthelot, L. Breen and W. Messing, Th\'{e}orie de Dieudonn\'{e}
%cristalline II. LNM 930, Springer-Verlag (1982).  

%\bibitem{BC}
%A. Borel and W. Casselman, editors,  \emph{Automorphic forms, representations and $L$-functions (Proc. Sympos. Pure Math., Oregon State Univ.,Corvallis, OR., 1977), Part 2}, volume 33 of Proc. Symp. Pure Math. Amer. Math. Soc., Providence, RI, 1979. 
\bibitem{Carter}
R.~Carter, {\it Finite groups of Lie type. Conjugacy classes and
complex characters.} Reprint of the 1985 original. Wiley Classics
Library. A Wiley-Interscience Publication. John Wiley \& Sons, Ltd.,
Chichester (1993), 544 pp.  

%\bibitem{CF}W.-S.~Cassels and A.~Fr\"ohlich, editors,   Algebraic Number Theory 
  % (Proc. Instructional Conf.,
 % Brighton, 1965),  Academic Press, 366 pp.    


\bibitem{Curtis}
C. W. Curtis, Modular representations of finite groups with split $(B,
N)$-pairs. In \emph{Seminar on Algebraic Groups and Related Finite
  Groups (The Institute for Advanced Study, Princeton, N.J.,
  1968/69)},  Springer, Berlin (1970),  57--95. 
%\bibitem{deligne:travaux} P. Deligne, Travaux de Shimura. {\it S\'em. Bourbaki Exp.}~{389}
 % (1970/71), 123--165.  \LNM{244}, 1971.
\bibitem{Deligne}
P. Deligne, Vari\'et\'e de Shimura: 
Interpr\'etation modulaire, et techniques de construction de mod\'eles
canoniques. {\it Automorphic forms, representations and
  $L$-functions.} 
%(Proc. Sympos.Pure Math., Oregon State
%  Univ. Corvallis, 1977), Part 2,} 
Proc. Sympos. Pure Math. 
  {\bf 33}, Part 2 (1979), 247--289.  
%In Borel and Casselman \cite{BC}. 

\bibitem{Ekedahl}
T. Ekedahl, On supersingular curves and supersingular abelian varieties. \emph{Math. Scand.}~{\bf 60} (1987), 151--178.

\bibitem{Ghitza}
A. Ghitza, Hecke eigenvalues of Siegel modular forms (mod $p$) and of
algebraic modular forms. \emph{J. Number Theory}~{\bf 106},
(2004), no. 2, 345--384. 

\bibitem{Ghitza2}
A. Ghitza,  Upper bound on the number of systems of Hecke eigenvalues
for Siegel modular forms (mod $p$). \emph{Int. Math. Res. Not.} (2004), no. 55, 
2983--2987.

% Upper Bound on the Number of Systems of Hecke Eigenvalues
% for Siegel Modular Forms (mod $p$), ??? 

\bibitem{GK}
W. Goldring and J.-S. Koskivirta, 
Strata Hasse invariant, Hecke algebras and Galois representations. \emph{Invent. Math.} {\bf 217}  (2019), no. 3, 887--984.  

\bibitem{Gross}
B. H.\ Gross, Algebraic modular forms. \emph{Israel J. Math.}~{\bf
  113} (1999), 61--93.  

\bibitem{Gross:2019} B. H. Gross, Incoherent definite spaces and
  Shimura varieties, Preprint (2019). 
  
 % \bibitem{EGA4} 
 %   A.~Grothendieck and J.~Dieudonn\'e: {\it El\'ements de
 %   G\'eom\'etrie Alg\'ebrique : IV. \'{E}tude locale des sch\'{e}mas et des morphismes de sch\'{e}mas}, Publ. I.E.H.S. {\bf
%    20,24,28,32} (1964-1967).
\bibitem{faltings-chai} G. Faltings and C.-L.~Chai, 
  {\it Degeneration of abelian varieties}. Ergebnisse der Mathematik 
  und ihrer Grenzgebiete (3), 22. Springer-Verlag, Berlin 
  (1990), xii+316 pp. 
%\bibitem{Hartshorne}
%R. Hartshorne, Algebraic Geometry. Graduate Texts in Mathematics,
%Vol. 52, Springer, 1977.

\bibitem{HP}
B. Howard and G. Pappas, Rapoport-Zink spaces for spinor groups,
\emph{Compos. Math.} {\bf 153} (2017), no. 5, 1050--1118.

%\bibitem{Humphreys}
%J. Humphreys, Modular Representations of finite Groups of Lie Type. 
% London Mathematical Society Lecture Note Series, 326. Cambridge
% University Press, Cambridge, 2006.
 
\bibitem{Jantzen}
J. Jantzen, Representations of algebraic groups. Mathematical Surveys
and Monographs, {\bf 107}, American Mathematical Society, Providence, RI 
(2003), 576 pp.

\bibitem{Kisin1}
M. Kisin, Integral models for Shimura varieties of abelian type. 
\emph{J. Amer. Math. Soc.}~{\bf 23} (2010), no. 4, 967--1012.

\bibitem{Kisin2}
M. Kisin, Mod $p$ points on Shimura varieties of abelian type. 
\emph{J. Amer. Math. Soc.}~{\bf 30} (2017), no. 3, 819--914.

\bibitem{Kottwitz85}
R. E. Kottwitz, Isocrystals with additional structure.
\emph{Compos. Math.}~{\bf 56}  (1985), 201--220.

\bibitem{Kottwitz}
R. E. Kottwitz, Points on some Shimura varieties over finite fields.  
\emph{J. Amer. Math. Soc.}~{\bf 5} (1992), 373--444.

\bibitem{Lan}
K.-W. Lan, Arithmetic compactification of PEL type Shimura varieties.
London Mathematical Society Monographs Series, 36. 
Princeton University Press, Princeton, NJ (2013). 

\bibitem{Lan2} 
K.-W. Lan, Toroidal Compactifications of PEL-type Kuga-families. 
\emph{Algebra Number Theory} $\mathbf{6}$ (2012), no. 5, 885--966. 

\bibitem{Lan3}
K.-W. Lan, Higher Koecher's principle.
\emph{Math. Res. Lett.}~$\mathbf{23}$ (2016), no. 1, 163--199. 

%\bibitem{LS}
%K.-W. Lan and J. Suh, Vanishing theorems for torsion automorphic
%sheaves on compact PEL-type Shimura varieties. \emph{Duke
%Math. J.}~$\mathbf{161}$ (2012), no. 6, 1113--1170. 

\bibitem{Lee}
D.-Y. Lee, Nonemptiness of Newton strata of Shimura varieties of Hodge
type.  \emph{Algebra Number Theory}~{\bf 12} (2018), no. 2, 259--283. 

\bibitem{Madapusi}
K. Madapusi-Pera, 
Toroidal compactifications of integral models of Shimura varieties of
Hodge type, \emph{Ann. Sci. Ec. Norm. Super. (4)}~{\bf 52} (2019),
no. 2, 393--514. 

%\bibitem{Manin}
%Y. Manin, Theory of commutative formal groups over fields of finite
%characteristic. \emph{Russian Math. Surveys}~{\bf 8} (1963), 1--80.
%\bibitem{mumford:av} D. Mumford, {\it Abelian Varieties.} Oxford
 % University Press (1974).

%\bibitem{Oda}
%T. Oda, The first de Rham cohomology group and Dieudonn\'{e} modules. 
%\emph{Ann. scient. \'{E}c. Norm. Sup. (4)}~{\bf 2} (1969), 63--135.

%\bibitem{Ogus} A.~Ogus, Supersingular K3 crystals.  {\it Journ\'ees
%de  G\'eom\'etrie. Alg\'ebrique de Rennes} (1978), Vol
%II, pp.  3--86, Ast\'erisque, no.  64. Soc. Math. France, Paris, 1979. 

%\bibitem{oort:product}
% F. Oort, Which abelian surfaces are products of
%  elliptic curves? \mathann~{\bf 214} (1975), 35--47.
\bibitem{RZ}
M. Rapoport and Th. Zink, Period spaces for $p$-divisible groups. 
Ann.~Math.~Studies 141, Princeton Univ. Press (1996).
\bibitem{Reduzzi}
D. A. Reduzzi, Hecke eigensystems for (mod $p$) modular forms of PEL
type and algebraic modular forms. \emph{Int. Math. Res. Not.}
(2013), no. 9, 2133--2178.
\bibitem{Richen}
F. A. Richen, Blocks of defect 
zero of split $(B, N)$ pairs.  
\emph{J. Algebra}
{\bf 21} (1972), 275--279.
\bibitem{SerreL}
J.-P. Serre, {\it{Linear representations of finite groups}}. 
%Translated from the second French edition by Leonard L. Scott. 
Graduate Texts in
Mathematics, Vol. 42. Springer-Verlag, New York-Heidelberg (1977).   
\bibitem{Serre}
J.-P. Serre, Two letters on quaternion and modular forms (mod
$p$). With introduction, appendix and 
references by R. Livn\'{e}. 
\emph{Israel J. Math.}~{\bf 95} (1996), 281--299. 

%\bibitem{SZ}
%X. Shen and C. Zhang, 
%Stratifications in good reductions of Shimura varieties of abelian type. 
%arXiv:1707.00439.
%\bibitem{shimura:ann1963} G.~Shimura, On analytic families of
%  polarized abelian varieties and automorphic functions. {\it
%  Ann. Math.}~{\bf 78} (1963), 149--192. 

%\bibitem{Shimura}
%G. Shimura, Some exact formulas for quaternion unitary groups.
%\emph{J. Reine Angew. Math.}~{\bf 509} (1999), 67--102.
%\bibitem{Silverman} J.~Silverman, {\it The arithmetic of elliptic
  %  curves.} {\it Graduate Texts in Mathematics}, {\bf
 %   106}. Springer-Verlag, New York, 1992. xii+400 pp. 
 \bibitem{TY}
Y. Terakado and C.-F. Yu, Hecke eigensystems of automorphic forms  (mod $p$) of Hodge type and algebraic modular forms. 
arXiv:2006.14342v1. 
\bibitem{VW}
E. Viehmann and T. Wedhorn, 
Ekedahl-Oort and Newton strata for Shimura varieties of PEL type. 
\emph{Math. Ann.}~{\bf 356} (2013), no. 4, 1493--1550. 

\bibitem{xue-yu:ppas} J.~Xue and C.-F. Yu, 
Polarized superspecial simple abelian  surfaces with real Weil numbers. 
arXiv:2009.02729.

%\bibitem{Yu4}
%C.-F. Yu, Lifting abelian varieties with additional structures. 
%\emph{Math. Z.}~{\bf 242} (2002), no. 3, 427--441. 

\bibitem{2003Fourier}
C.-F. Yu, On reduction of Hilbert-Blumenthal varieties.
\emph{Ann. Inst. Fourier (Grenoble)}~{\bf 53} (2003), no. 7, 2105--2154. 



\bibitem{slope}
C.-F. Yu, On the slope stratification of certain Shimura varieties. 
\emph{Math. Z.}~{\bf 251} (2005), no. 4,  859--873.



\bibitem{Yu2}
C.-F. Yu, An exact geometric mass
formula. \emph{Int. Math. Res. Not.},  
Article ID inn113 (2008), 11 pp.

%\bibitem{Yu}
%C.-F. Yu, Simple mass formulas on Shimura varieties of PEL-type. 
%\emph{Forum Math.}~{\bf 22}, (2010), no. 3,  565--582.

\bibitem{Yu3}
C.-F. Yu,  Non-emptiness of the basic locus of Shimura varieties. \emph{Oberwolfach Report}
No. 39/2015, ``Reductions of Shimura varieties". 
%\bibitem{Yu5}
%C.-F. Yu, On arithmetic of the  superspecial locus. \emph{Indiana
 % Univ. Math. J.}~{\bf 67} (2018), no. 4, 1341--1382.


\bibitem{yu:D}
C.-F. Yu, Reduction of the moduli schemes of abelian varieties with
definite quaternion multiplication. To appear in
\emph{Ann. Inst. Fourier (Grenoble)}.  

\bibitem{XZ}
L. Xiao and X. Zhu, Cycles on Shimura varieties via geometric Satake.
arXiv:1707.05700.

% preprint. 

%\bibitem{Zarhin}
%J. G. Zarhin, Isogenies of abelian varieties over fields of finite characteristics. \emph{Math. USSR Sbornik}~{\bf 24} (1974), 451--461. 

\bibitem{Zhang}
C. Zhang, Ekedahl-Oort strata for good reduction of Shimura varieties
of Hodge type. \emph{Canad. J. Math.}~{\bf 70} (2018), no.2, 451--480. 

\bibitem{Zhang2}
C. Zhang, 
Stratifications and foliations for good reductions of Shimura
varieties of Hodge type. arXiv:1512.08102.
\end{thebibliography}
\end{document}